\documentclass[12pt]{amsart}
\usepackage{amsmath,amssymb,amsfonts,amsthm,amsopn,mathrsfs}

\usepackage{graphicx,tikz}
\usepackage[all]{xy}
\usepackage{amsmath}
\usepackage{multirow, longtable, makecell, caption, array}
\usepackage{amssymb}
\usepackage{mathtools}
\usepackage{pb-diagram}
\usepackage{cellspace} %
\usepackage{enumerate}
\newcommand{\aaf}{A_a^{\varphi_1,\varphi_2}}

\newtheorem{tm}{Theorem}[section]
\newtheorem{lemma}[tm]{Lemma}

\newtheorem{theorem}{Theorem}[section]
\newtheorem{corollary}[theorem]{Corollary}
\newtheorem{definition}[theorem]{Definition}

\newtheorem{proposition}[theorem]{Proposition}
\newtheorem{remark}[theorem]{Remark}
\newtheorem*{thm*}{Theorem}
\newtheorem*{cor*}{Corollary}
\newcommand{\beqa}{\begin{eqnarray*}}
	\newcommand{\eeqa}{\end{eqnarray*}}

\DeclareMathOperator*{\Opt}{Op_\tau}
\DeclareMathOperator*{\Opz}{Op_{0}}
\DeclareMathOperator*{\Opo}{Op_{1}}

\DeclareMathOperator*{\Opw}{Op^{\textit{w}}}
\DeclareMathOperator*{\OpW}{Op_{\textit{W}}}
\newcommand{\field}[1]{\mathbb{#1}}
\newcommand{\bR}{\field{R}}
\newcommand{\bN}{\field{N}}
\newcommand{\bZ}{\field{Z}}
\newcommand{\bC}{\field{C}}

\def\al{\alpha}
\def\be{\beta}
\def\ga{\gamma}

\def\G{\mathcal{G}}

\def\la{\lambda}

\def\om{\omega}

\def\cF{\mathcal{F}}
\def\cS{\mathcal{S}}

\def\cA{\mathcal{A}}

\def\cC{\mathcal{C}}

\def\cT{\mathcal{T}}

\def\cA{\mathcal{A}}
\def\a{\aleph}
\def\hf{\hat{f}}
\def\hg{\hat{g}}

\def\rd{\bR^d}

\def\rdd{{\bR^{2d}}}
\def\zdd{{\bZ^{2d}}}

\def\lrd{L^2(\rd)}

\def\zd{\bZ^d}

\def\intrd{\int_{\rd}}
\def\intrdd{\int_{\rdd}}
\def\iintrdd{\iint_{\rdd}}

\def\R{\right)}

\def\<{\left<}
\def\>{\right>}

\def\inv{^{-1}}

\def\mv1{M_v^1}

\def\phas{(x,\omega )}

\newcommand{\abs}[1]{\lvert#1\rvert}
\newcommand{\norm}[1]{\lVert#1\rVert}
\def\o{\omega}
\def\a{\alpha}
\def\b{\beta}

\def\i{\infty}

\def\R{\mathbb{R}}
\def\Ren{\mathbb{R}^d}

\def\sch{\mathcal{S}}

\def\Fur{\mathcal{F}}

\def\f{\varphi}

\def\gaw{A_a^{\f_1,\f_2}}
\def\gawtau{A_{a, \tau} ^{\f_1,\f_2}}

\def\Sn2{S_{2}(L^{2}(\Ren))}
\def\S1{S_{1}(L^{2}(\Ren))}
\def\sig00{\sigma_{0,0}}

\def\la{\langle}
\def\ra{\rangle}

\usepackage{xcolor,soul}
\usepackage{calc}

\usepackage{cite}
\usepackage{xspace}

\newcommand{\rr}[1]{\mathbb R^{#1}}

\newcommand{\ep}{\varepsilon}

\newcommand{\gel}{\mathcal{S}^{\left(1\right)}}
\newcommand{\ult}{\mathcal{S}^{\left(1\right)'}}

\usepackage{enumitem}

\begin{document}

	\begin{abstract}
We consider time-frequency localization operators $\gaw $ with symbols $a$ in the wide weighted modulation space
$ M^\infty_{w}(\rdd)$, and windows $ \varphi_1, \varphi_2 $ in the Gelfand-Shilov space $\gel(\rd)$. If the weights under consideration are of ultra-rapid growth, we prove that the eigenfunctions of $\gaw $ have appropriate sub\-exponential decay in phase space, i.e. that they belong to the Gefand-Shilov space $ \cS^{(\gamma)} (\rr d) $, where the parameter $\gamma \geq 1 $ is related to the growth of the considered weight.
An important role is played by $\tau$-pseudodifferential operators $\Opt(\sigma)$. In that direction we show convenient continuity properties of
$\Opt(\sigma)$ when acting on weighted modulation spaces. Furthermore, we prove
subexponential decay and regularity properties of the eigenfunctions of  $\Opt(\sigma)$ when the symbol $\sigma$ belongs to a
modulation space with appropriately chosen weight functions. As a tool we also prove new convolution relations for (quasi-)Banach weighted modulation spaces.
	\end{abstract}
	
	\title[Subexponential decay and regularity estimates]{Subexponential decay and regularity estimates for eigenfunctions of localization operators}
	\author[F. Bastianoni]{Federico Bastianoni}
	\address{Dipartimento di Scienze Matematiche, Politecnico di Torino, corso
		Duca degli Abruzzi 24, 10129 Torino, Italy}
	\email{federico.bastianoni@polito.it}
	\thanks{}
	\author[N. Teofanov]{Nenad Teofanov}
	\address{Department of Mathematics and Informatics,
	University of Novi Sad, Serbia}
	\email{nenad.teofanov@dmi.uns.ac.rs}
	\thanks{}
	
	\date{}

\maketitle

\section{Introduction }

It is known that the eigenfunctions of time-frequency localization operators, also known as Daubechies operators and denoted by $\gaw $,
with Gaussian windows
$$\varphi_1 (t) = \varphi_2 (t) =  \pi^{-d/4} \text{exp} (-t^2/2) \;\;\;
\text{and with a radial symbol} \;\;\; a \in L^1 (\mathbb{R}^{2d} ),
$$
are Hermite functions, i.e. they have superexponential decay in phase space, \cite{DB1}. In a different terminology, 
those eigenfunctions belong to the smallest projective  Gelfand-Shilov space $\cS^{\{ 1/2 \}} (\rr d) $ (cf. Definition \ref{de:gsspaces}).

\par

The investigations in \cite{DB1} are motivated by some questions in signal analysis. The same type of operators (under the name anti-Wick operators) is used to study the quantization problem in quantum mechanics, cf. \cite{Berezin71}. In abstract harmonic analysis, localization operators on a locally compact group $G$ and Lebesgue spaces $ L^p (G)$, $ 1\leq p \leq \infty $,
were studied in \cite{WongLocalization}. We also mention their presence in the form of Toeplitz operators in complex analysis \cite{BergCob87}.
Here we do not intend to discuss different manifestations of localization operators and refer to e.g. \cite{Englis} for a survey.

\par

In the framework of time-frequency analysis an important step forward in the study of localization operators 
was made by the seminal paper \cite{EleCharly2003}. Thereafter the subject is considered by many authors including \cite{BG2015,BCG02,ECSchattenloc2005,medit,CharlyToft2011,Nenad2016,Toft_2012_Bergmann} where among others, one can find different continuity, Schatten class and lifting properties of localization operators.
The time-frequency analysis approach   is based on the use of modulation spaces as appropriate functional analytic framework. Another issue established in  \cite{EleCharly2003, BCG02} is the identification of localization operators as Weyl pseudodifferential operators.

\par

The focus of this paper is to consider the properties of {\em  eigenfunctions} of compact localization operators. Our investigations are inspired by the recent work \cite{BasCorNic2019}. Indeed, there it is shown that if the symbol $a$ belongs to the modulation space
$ M^\infty_{v_s\otimes1}(\rdd)$, $ s>0$ (see Definition \ref{Def-UltraMod}) and $ \varphi_1, \varphi_2 \in  \sch(\Ren)$, then the eigenfunctions of $\gaw $ are actually Schwartz functions. Moreover, similar result is proved for the Weyl pseudodifferential operators whose symbol belongs to
$ M^{\infty,1} _{v_s\otimes v_t}(\rdd)$, for some $ s>0$ and every $t>0$, cf. \cite[Proposition 3.6]{BasCorNic2019}.
Here $ v_s (z) = (1 + |z|^2)^{s/2} $, $ s \in \mathbb{R}, $ $ z \in \rd$.

\par

We extend the scope of \cite{BasCorNic2019} by considering a more general class of weights, which contains the weights of subexponential growth,
apart from polynomial type weights. As explained in \cite{CharlyToft2013}, replacing polynomial weights with weights of faster growth at infinity is not a mere routine. Indeed, to treat weights of ultra-rapid growth it is necessary to replace the most common framework of the Schwartz space of test functions and its dual space of tempered distributions by the more subtle family of Gelfand-Shilov spaces and their duals spaces of ultra-distributions, cf. \cite{GS,CCK,Gram,NR,Pi88,Nenad2015,Toft_2012_Bergmann}.  To underline this difference, we refer to ultra-modulation spaces when modulation spaces are allowed to contain such ultra-distributions.

\par

One of the main tools in our analysis is the (cross-)$\tau$-Wigner distribution $ W_\tau(f,g)$, $f,g \in L^2 (\rd) $, see Definition \ref{def:tauwigner}. The relation between $ W_\tau(f,g)$ and another relevant time-frequency representation, namely
the short-time Fourier transform $V_g (f)$ (cf. Lemma \ref{lm:wignerstft}) serves as a bridge between properties of modulation
spaces and $\tau$-pseudodifferential operators. 
More precisely, we extend the recent result
\cite[Theorem 3.3]{BasCorNic2019} to a more general class of operators and weights (Theorem \ref{Th-ExtensionWeylOp}). 
Although this result follows from \cite[Theorem 3.1]{ToftquasiBanach2017} our proof is more elementary and independent.

\par

Our first main result concerns decay properties of the eigenfunctions of $\tau$-pse\-u\-do\-differen\-ti\-al operators. In fact, by using iterated
actions of the operator we conclude that its  eigenfunctions belong to the Gelfand-Shilov space  $\cS^{(\ga)}(\rd)$ (Theorem \ref{Pro-SmoothnessEigenfuctionsWeyl}). As already mentioned, 
this gives an information about regularity and decay properties of eigenfunctions which can not be captured  within the Schwartz class.

\par

Finally, we use Theorem \ref{Pro-SmoothnessEigenfuctionsWeyl} and convolution relation for modulation spaces (Proposition \ref{Pro-ConvRel}) to show that the eigenfunctions of localization operators $\gaw $
have appropriate subexponential decay in phase space if  $a \in M^\infty_{w}(\rdd)$,  $ \varphi_1, \varphi_2 \in\gel(\rd)$, and if
$w$ is of a certain  ultra-rapid growth. We use the representation of localization operators as pseudodifferential operators. Evidently, the Weyl form of localization operators suggests to introduce and consider $\tau$-localization operators by using 
$\tau$-pseudodifferential operators and the (cross-)$\tau$-Wigner distribution. 
However, it turns out the such approach does not extend the class of localization operators given by Definition \ref{def:locop}
(cf. Proposition  \ref{pr:taulocop}).

\par

We end this introduction with a brief report of the content of the paper. In Preliminaries we collect relevant
background material.  Apart from the review of known results it contains some new results or proofs (Lemma \ref{Lem-condition(sub-)exp-verified}, Proposition \ref{pr:taulocop}, Proposition \ref{Pro-ConvRel}). In Section \ref{main results}
we prove our main results: continuity properties of $\tau$-pse\-u\-do\-differen\-ti\-al operators on modulation spaces,
estimates for eigenfunctions of $\tau$-pseudodifferential operators, and
decay and smoothness properties of eigenfunctions of localization operators.
Appendix contains the proofs of two auxiliary technical results.

\par

\subsection{Notation} We denote the Euclidean scalar product on $\Ren$ by $xy\coloneqq x\cdot y$ and the Euclidean norm by $\abs{x}\coloneqq\sqrt{x\cdot x}$. We put $\bN_0\coloneqq\bN\cup\{0\}$. $A\lesssim B$ means that for given constants $A$ and $B$ there exists a constant $c>0$ independent of $A$ and $B$ such that $A\leq cB$, and we write $A\asymp B$ if both $A\lesssim B$ and $B\lesssim A$.
We define the involution $g^*$ of a function $g$ by $g^*(t) \coloneqq \overline{g(-t) }$.
Given a function $f$ on $\rd$ its Fourier transform is normalized to be
\[
\Fur f(\o)=\hf(\o)\coloneqq \int_{\rd} e^{-2\pi i x\o} f(x)\, dx,\qquad\o\in\rd.
\]

Given two spaces $A$ and $B$, we denote by $A\hookrightarrow B$ the continuous embedding of $A$ into $B$. $\sch(\Ren)$ denotes the Schwartz class and its topological dual, the space of tempered distributions, is indicated by  $\sch'(\Ren)$. By the brackets  $\la f,g\ra$ we mean the extension  of the $L^2$-inner product $\la f,g\ra\coloneqq\int f(t){\overline {g(t)}}\,dt$ to any dual pair.

Consider $0<p< \infty$ and a positive and measurable function $m$  on $\rd$, then $L^p_m(\rd)$ denotes the (quasi-)Banach space of measurable functions $f: \rd \to \bC$ such that 
$$\|f\|_{L^p_m}\coloneqq\left(\intrd |f(x)|^p m(x)^p\, dx\right)^{1/p}<+\infty,$$
modulus the equivalence relation $f\sim g$ $\Leftrightarrow$ $f(x)=g(x)\,$ for  a.e. $x$.
When $p=\infty$, $f\in L^\infty_m(\rd)$ if $\|f\|_{L^\infty_m }\coloneqq \; \text{ess sup}\;_{x\in\rd} |f(x)| m(x)<+\infty$, up to the equivalence relation defined above. If $ m\equiv 1,$ we use abbreviated notation $  L^p (\rd) = L^p _1(\rd)$. If the restriction of $f$ to any compact set belongs to $L^p (\rd)$, then we write $f \in  L^p _{loc }(\rd)$.

For given Hilbert space $H$ and compact operator $T$ on $H$ its singular values $\{s_k(T)\}_{k=1}^\infty$ are the eigenvalues of $(T^*T)^{1/2}$, which is a positive and self-adjoint operator. The Schatten class $S_p(H)$, with $0<p<\infty$, is the set of all compact operators on $H$ such that their singular values are in $\ell^p$. For consistency, we define $S_\infty(H)\coloneqq B(H)$, the set of all linear and bounded operators on $H$. We shall deal with $H=L^2(\rd)$.

By $\sigma_P(T)$  we denote the \emph{point spectrum} of the operator $T$. If $T$ is a compact mapping on $\lrd$ then the spectral theory for compact operators yields $\sigma(T)\smallsetminus\{0\}=\sigma_P(T)\smallsetminus\{0\}$, where $\sigma(T)$ is the spectrum of the operator. For compact operators on $\lrd$ we have $ 0 \in \sigma(T)$,  and  the point spectrum $\sigma_P(T)\smallsetminus\{0\}$ (possibly empty) is at most a countable set.

A function $f\in\lrd\smallsetminus\{0\}$ is  an \emph{eigenfunction} of the operator $T$ if there exists $\lambda\in \bC$ such that
$ T f=\lambda f.$ We are interested in the properties of eigenfuctions  of $\aaf$ related to eigenvalues $\lambda\in \sigma_P(\aaf)\smallsetminus\{0\}$, whenever $ \sigma_P(\aaf)\smallsetminus\{0\}\not=\emptyset$.

\section{Preliminaries}

In this section we collect background material and prove some auxiliary results.

\subsection{Weight functions}
By \textit{weight} $m$ on $\rd$ (or on $\zd$) we mean a positive function $m>0$ such that $m\in L^{\infty}_{loc}(\rd)$ and $1/m\in L^{\infty}_{loc}(\rd)$. A weight $m$ is said to be \textit{submultiplicative} if it is even and
\begin{equation*}
	m(x+y)\leq m(x)m(y),\qquad\forall\,x\in\rd.
\end{equation*}
Given a weight $m$ on $\rd$ and a positive function $v\in L^{\infty}_{loc}(\rd)$, we say that $m$ is $v$-\textit{moderate} if
\begin{equation*}
	m(x+y)\lesssim v(x)m(y),\qquad\forall\,x,y\in\rd.
\end{equation*}
Therefore submultiplicative weights are moderate and the previous inequality implies the following estimates:
\begin{equation*}\label{Eq-ConsequenceOfModerateness}
	v(-x)^{-1}\lesssim m(x)\lesssim v(x),\qquad \forall\,x\in\rd.
\end{equation*}
For a submultiplicative weight $v$ there are convenient ways to find smooth weights $v_0$ which are equivalent to $v$ in the sense that there is a constant $C>0$ such that
$$
C^{-1} v_0 \leq  v \leq C v_0,
$$
see e.g.\cite{CordRod2020, Grochenig_2007_Weight, Toft_2012_Bergmann}.

Next we introduce some weights which will be used in the sequel. Given $k,\ga>0$ we define
\begin{equation*}\label{Eq-(sub-)exp}
	w^\ga_k(x)\coloneqq e^{k\abs{x}^{1/\ga}},\qquad x\in\rd.
\end{equation*}
Sometimes we shall use the above expression for $k=0$ also, with obvious meaning. If $\ga > 1$ the above functions are called subexponential weights, and when $\ga=1$ we write $w_k$ instead of $w^1_k$. Note that (sub-)exponential weights $w^\ga_k$ are submultiplicative (this follows from \eqref{Eq-useful-ineq-1}).  When $0 <\ga<1$ we obtain weights of super-exponential growth at infinity. We shall work with the following weight classes defined for $\ga>0$:
\begin{align*}
	\mathscr{P}_E(\rd)&\coloneqq\{m\,\text{weight on}\,\rd\,|\, m\,\text{is $v$-moderate for some submultiplicative $v$} \},\\
	\mathscr{P}_{E,\ga}(\rd)&\coloneqq\{m\,\text{weight on}\,\rd\,|\,m\,\text{is $w^\ga_k$-moderate for some $k>0$}\},\\
	\mathscr{P}^0_{E,\ga}(\rd)&\coloneqq\{m\,\text{weight on}\,\rd\,|\,m\,\text{is $w^\ga_k$-moderate for every $k>0$}\}.
\end{align*}
For $0<\ga_2<\ga_1$ we have
\begin{equation*}
	\mathscr{P}^0_{E,\ga_1}\subseteq\mathscr{P}_{E,\ga_1}\subseteq\mathscr{P}^0_{E,\ga_2}\subseteq\mathscr{P}_E.
\end{equation*}
Moreover, for $0<\ga<1$ we have $\mathscr{P}_E=\mathscr{P}_{E,\ga}=\mathscr{P}^0_{E,\ga}$; see \cite[Remark 2.6]{CappToft2016} and \cite{Toft19}. In the next lemma we show that if $m\in\mathscr{P}_E$, then it is $w_k$-moderate fore some $k>0$ large enough. This implies $\mathscr{P}_E=\mathscr{P}_{E,1}$.

\begin{lemma}
	Let $m\in\mathscr{P}_E$. Then $m$ is $w_k$-moderate fore some $k>0$.
\end{lemma}

\begin{proof}
The lemma is folklore (\cite{Grochenig_2007_Weight, CappToft2016, Toft-ImageBargmann-2017,ToftquasiBanach2017}).
For the sake of completeness we report a self-contained proof following \cite{Grochenig_2007_Weight}.
By the hypothesis, we may assume that $m$ is moderate with respect to some continuous $v_0>0$: $m(x+y)\leq C v_0(x)m(y)$, $x,y\in\rd$. It follows that $\sup_{\abs{t}\leq1}C v_0(t)=e^a$ for some $a\in\mathbb{R}$. For any given $x,y\in\rd$ we choose $n\in\bN$ such that $n-1<\abs{x}\leq n$. Then
\begin{multline*}
m(x+y) = m\left(n\frac{x}{n}+y\right)  \leq C v_0\left(\frac{x}{n}\right)m\left((n-1)\frac{x}{n}+y\right)\\
\leq C^2v_0^2 \left(\frac{x}{n}\right) m\left((n-2)\frac{x}{n}+y\right) \leq \dots\\
\leq\left(Cv_0\left(\frac{x}{n}\right)\right)^n m(y) \leq e^{an}m(y)\\
<e^{a(\abs{x}+1)}m(y)=e^a e^{a\abs{x}}m(y), \;\;\; 			 x,y\in\rd.
\end{multline*}
The claim follows for $k>\max(0,a)$.
\end{proof}

\par

We remark that $\mathscr{P}_E$ contains the weights of \textit{polynomial type}, i.e. weights moderate with respect to some polynomial.

In the sequel $\mathscr{P}^\ast_{E,\ga}$ means $\mathscr{P}_{E,\ga}$ or $\mathscr{P}^{0}_{E,\ga}$. The following lemma follows by easy calculations and we leave the proof for the reader (see also \cite{Toft_2012_Bergmann}). Observe that due to the equality $\mathscr{P}_{E,1}=\mathscr{P}_{E,\ga}=\mathscr{P}^0_{E,\ga}$, $0<\ga<1$, it is sufficient to consider $\ga\geq 1$.

\begin{lemma}\label{Lem-Weight1}
Consider $\ga>0$. Then $\mathscr{P}^{\ast}_{E,\ga}(\rd)$ is a group under the pointwise multiplication and with the identity $m\equiv 1$.
\end{lemma}

Given a function $f$ defined on $\rdd$ we denote its restrictions to $\rd\times\{0\}$ and $\{0\}\times\rd$ as follows:
\begin{equation}\label{Eq-Restrictions}
	f|_1(x)\coloneqq f(x,0),\qquad f|_2(\o)\coloneqq f(0,\o),\qquad x,\o\in\rd.
\end{equation}
Given two functions $g,h$ defined on $\rd$  their \textit{tensor product} is the function on $\rdd$ defined in the following manner:
\begin{equation*}\label{Eq-TensorProducts}
	g\otimes h\phas\coloneqq g(x)h(\o),\qquad\phas\in\rdd.
\end{equation*}
The families $\mathscr{P}^{\ast}_{E,\ga}$ turn out to be closed under restrictions and tensor products in the sense of the following lemma.
The proof is omitted, since it follows from definitions and properties of the Euclidean norm.

\par
\begin{lemma}\label{Lem-Weight2} Consider $\ga>0$:
	\begin{enumerate}[label=(\roman*)]
	\item if $m\in\mathscr{P}^{\ast}_{E,\ga}(\rdd)$, then $m|_1,m|_2\in\mathscr{P}^{\ast}_{E,\ga}(\rd)$;
	\item if $m,w\in\mathscr{P}^{\ast}_{E,\ga}(\rd)$, then $m\otimes w\in\mathscr{P}^{\ast}_{E,\ga}(\rdd)$.
	\end{enumerate}
\end{lemma}

Next we exhibit a lemma which will play a key role in the sequel, see Proposition \ref{Pro-SmoothnessEigenfuctionsWeyl}.
The proof is given in the appendix.

\begin{lemma}\label{Lem-condition(sub-)exp-verified}
	Consider $\ga\geq1,\,r,s\geq0$, $\tau\in[0,1]$ and
	\begin{equation}\label{Eq-condition-on-t}
		t\geq
		\begin{cases}
		r+s\tau^{1/\ga}\qquad\qquad&\text{if}\qquad1/2\leq\tau\leq1,\\
		r+s(1+\tau^2)^{1/2\ga}\quad&\text{if}\qquad0\leq\tau<1/2.
		\end{cases}
	\end{equation}
	Then for every $x,\o,y,\eta\in\rd$ the following estimate holds true:
	\begin{equation}\label{Eq-condition(sub-)exp-verified}
		\frac{w^\ga_{r+s}\phas}{w^{\ga}_r(y,\eta)}\leq w^\ga_s\otimes w^\ga_t\Big(\big((1-\tau)x+\tau y,\tau\o+(1-\tau)\eta\big),\big(\o-\eta,y-x\big)\Big).
	\end{equation}
\end{lemma}


We finish this subsection by introducing some polynomial weights which will be used in Theorem \ref{Th-Schatten-tau-pseudo} and Lemma \ref{Lem-inclusion-for-compactness}. Let $\tau\in[0,1]$ and  $u\geq0$, then we define  the weights of polynomial type
\begin{align}
v_u\phas&\coloneqq\<\phas\>^u=(1+\abs{\phas}^2)^{u/2}, \;\;\; \phas \in\rdd,
\label{Eq-DefPolynomialWeights}\\
m^\tau_u(\phas,(y,\eta))&\coloneqq(1+\abs{x-\tau\eta}+\abs{\o+(1-\tau)y})^u,
\;\;\; \phas,(y,\eta)\in\rdd.
\label{Eq-definition-m-u-tau}
\end{align}

\begin{remark}\label{Rem-m-tau_u-lesssim-v_u}
If $v_u $ and $m^\tau_u$ are given by \eqref{Eq-DefPolynomialWeights} and \eqref{Eq-definition-m-u-tau} respectively,
then we notice that
	\begin{equation*}
	m^\tau_u(\phas,(y,\eta))\lesssim v_u\otimes v_u(\phas,(y,\eta)),\;\;\; \forall\,\phas,(y,\eta)\in\rdd.
	\end{equation*}
	which will be used in Lemma \ref{Lem-inclusion-for-compactness}. Indeed:
	\begin{multline*}
	m^\tau_u(\phas,(y,\eta))=(1+\abs{x-\tau\eta}+\abs{\o+(1-\tau)y})^u\\
\lesssim(1+(\abs{x}+\abs{\tau\eta})^2+(\abs{\o}+\abs{(1-\tau)y})^2)^{u/2}\\
\lesssim(1+\abs{x}^2+\tau^2\abs{\eta}^2+\abs{\o}^2+(1-\tau)^2\abs{y}^2)^{u/2}
\lesssim
(1+\abs{\phas}^2+\abs{(y,\eta)}^2)^{u/2}\\
\leq(1+\abs{\phas}^2+\abs{(y,\eta)}^2+\abs{\phas}^2\abs{(y,\eta)}^2)^{u/2}\\
=(1+\abs{\phas}^2)^{u/2}(1+\abs{(y,\eta)}^2)^{u/2}
= v_u\otimes v_u(\phas,(y,\eta)).
	\end{multline*}
\end{remark}

\subsection{Spaces of sequences}
Given $0<p,q\leq \infty$ and $m\in \mathscr{P}_E(\zdd)$, $\ell^{p,q}_m(\zdd)$ is the set of all sequences $a=(a_{k,n})_{k,n\in\zd}$ such that the (quasi-)norm
$$\|a\|_{\ell^{p,q}_m}\coloneqq\left(\sum_{n\in\zd}\left(\sum_{k\in\zd}|a_{k,n}|^p m(k,n)^p\right)^{\frac qp}\right)^{\frac 1q}
$$
(with obvious changes for $p=\infty$ or $q=\infty$) is finite.\\
When $p=q$ we recover the standard spaces of sequences $\ell^{p,p}_m(\zdd)=\ell^p_m(\zdd)$.\\
In the following proposition we collect some properties that we shall use later on, see \cite{Galperin2014,Galperin2004}.
\begin{proposition}
\begin{itemize}
	\item[(i)] \emph{Inclusion relations}: Consider $0<p_1\leq p_2 \leq\infty$ and let $m$ be any positive weight function on $\zd$. Then
	\begin{equation*}
		\ell^{p_1}_m(\zd)\hookrightarrow\ell^{p_2}_m(\zd).
	\end{equation*}
	\item[(ii)] \emph{Young's convolution inequality}: Consider $m,v\in\mathscr{P}_E(\zd)$ such that $v$ is submultiplicative and $m$ is $v$-moderate, $0<p,q,r\leq \infty$ with
	\begin{equation*}\label{Yrgrande1}
	\frac1p+\frac1q=1+\frac1r, \quad \mbox{for}\quad 1\leq r\leq \infty
	\end{equation*}
	and
	\begin{equation*}\label{Yrminor1}
	p=q=r, \quad \mbox{for}\quad 0<r<1.
	\end{equation*}
	Then for all $a\in \ell^p_m(\zd),\,b\in\ell^q_v(\zd)$, we have $a\ast b\in \ell^r_m(\zd)$, with
	\begin{equation*}
	\|a\ast b\|_{\ell^r_m}\leq C \|a\|_{\ell^p_m}\|b\|_{\ell^q_v},
	\end{equation*}
	where the constant $C>0$ is independent of $p,q,r$, $a$ and $b$. If $m\equiv v\equiv1$, then $C=1$.
	\item[(iii)] \emph{H\"{o}lder's inequality}: Let $m$ be any positive weight function on $\zd$ and $0<p,q,r\leq \infty$ such that $1/p+1/q=1/r$. Then
	\begin{equation*}\label{ptwellp}
	\ell^p_m (\zd)\cdot \ell^q_{1/m}(\zd)\hookrightarrow \ell^r(\zd).
	\end{equation*}
\end{itemize}
\end{proposition}

\subsection{Gelfand-Shilov spaces} \label{subs:gsspace}

Let $h,\gamma,\tau > 0$ be fixed. Then $S _{\tau;h}^{\gamma}(\rr d)$
is the Banach space of all $f\in C^\infty (\rr d)$ such that
\begin{equation}\label{gfseminorm}
\| f \|_{S _{\tau;h}^{\gamma}}  \coloneqq
\sup_{p,q \in \mathbb{N} ^d _0} \sup_{x \in \rr d} \frac {|x^p \partial ^q 	f(x)|}{h^{|p|  + |q|} |p| !^{\tau} \, |q| !^\gamma}<+\infty,
\end{equation}
endowed with the norm \eqref{gfseminorm}.

\begin{definition} \label{de:gsspaces}
Let $\gamma,\tau > 0$. The Gelfand-Shilov spaces $\cS _{\tau}^{\gamma}(\rr d)$ and
$\Sigma _{\tau}^{\gamma}(\rr d)$ are defined as unions and intersections of $S _{\tau;h}^{\gamma}(\rr d)$ with respective
inductive and projective limit topologies:
\begin{equation*}\label{GSspacecond1}
\cS _{\tau} ^{\gamma}(\rr d) \coloneqq \bigcup _{h>0}S _{\tau;h}^{\gamma}(\rr d)
\quad \text{and} \quad \Sigma _{\tau}^{\gamma}(\rr d) \coloneqq\bigcap _{h>0}
S _{\tau;h}^{\gamma}(\rr d).
\end{equation*}
\end{definition}
Note that
$\Sigma _{\tau}^{\gamma}(\rr d)\neq \{ 0\}$ if and only if
$\tau +\gamma \ge 1$ and $(\tau,\gamma )\neq (1/2,1/2)$, and
$\cS _{\tau}^{\gamma}(\rr d)\neq \{ 0\}$ if and only
if $\tau+\gamma \ge 1$, see \cite{GS, Pi88}.
For every $\tau,\gamma, \ep >0$ we have
\begin{equation}\label{Eq:GSEmbeddings}
\Sigma _\tau ^\gamma (\rr d)
\hookrightarrow
\cS _\tau ^\gamma (\rr d)
\hookrightarrow
\Sigma _{\tau+\ep}^{\gamma +\ep}(\rr d)
\hookrightarrow
\cS (\rr d).
\end{equation}
If $\tau+\gamma \ge 1$, then the last two inclusions in \eqref{Eq:GSEmbeddings} are dense,
and if in addition $(\tau,\gamma )\neq (1/2,1/2)$
then the first inclusion in \eqref{Eq:GSEmbeddings} is dense.
Moreover, for $\gamma<1$ the elements of $\cS _{\tau}^{\gamma}(\rr d)$ can be extended to entire
functions on $\mathbb{C}^d$ satisfying suitable exponential bounds, \cite{GS}.\\
In the sequel we will also use the following notations:
\begin{equation*}
\cS^{(\gamma)} (\rr d) \coloneqq \Sigma _\gamma ^\gamma (\rr d),
\qquad
\cS^{\{\gamma\}} (\rr d) \coloneqq \cS_\gamma ^\gamma (\rr d)\;\;\; \text{and} \;\;\; \cS^{\ast} (\rr d),
\end{equation*}
where $ * $ stands for $ (\gamma) $ or $ \{\gamma\}.$

\begin{definition}
The \textit{Gelfand-Shilov distribution spaces} $(\cS _{\tau}^{\gamma})'(\rr d)$
and $(\Sigma _{\tau}^{\gamma})'(\rr d)$ are the projective and inductive limit
respectively of $(S _{\tau;h}^{\gamma})'(\rr d)$, the topological dual of  $S _{\tau;h}^{\gamma} (\rr d)$:
\begin{equation*}
\label{GSspacedistr}
(\cS _{\tau}^{\gamma})'(\rr d) \coloneqq \bigcap _{h>0}(S _{\tau;h}^{\gamma})'(\rr d)\quad
\text{and}\quad (\Sigma _{\tau}^{\gamma})'(\rr d) \coloneqq\bigcup _{h>0}(S _{\tau;h}^{\gamma})'(\rr d).
\end{equation*}
\end{definition}
It follows that $ \cS '(\rr d)\hookrightarrow
(\cS _\tau^\gamma)'(\rr d)$ when $\tau+\gamma \ge 1$, and if in addition
$(\tau,\gamma )\neq (1/2,1/2)$, then $(\cS _\tau^\gamma)'(\rr d)
\hookrightarrow (\Sigma _\tau^\gamma)'(\rr d)$.

\par

The Gelfand-Shilov spaces enjoy beautiful symmetric characterizations which also involve the Fourier transform of their elements.
The following result has been reinvented several times, in similar or analogous terms, see  \cite{CCK,GZ,KPP,NR}.

\begin{theorem} \label{simetrija} Let  $\gamma, \tau \geq 1/2$.
	The following conditions are equivalent:
	\begin{itemize}
		\item[(i)] $f \in {\mathcal S}^{\gamma}_{\tau} (\rr d)$ (resp. $ f \in \Sigma ^{\gamma} _{\tau} (\rr d) $);
		\item[(ii)]  There exist (resp. for every) constants $A,B>0$ such that
		$$ \| x^{p}  f (x)\|_{L^\infty} \lesssim A^{|p|} |p| ! ^{\tau}  \quad\mbox{and}\quad  \| \o^{q}  \hat{f} (\o) \|_{L^\infty} \lesssim B^{|q|} |q| ! ^{\gamma},
		\quad\forall p,q\in \mathbb{N} ^d _0; $$
		\item[(iii)]  There exist (resp. for every) constants  $A,B>0$ such that
		$$ \| x^{p}  f(x) \|_{L^\infty} \lesssim A^{|p|} |p| ! ^{\tau}  \quad\mbox{and}\quad  \| \partial^{q}  f (x)\|_{L^\infty} \lesssim B^{|q|} |q| !^{\gamma },
		\quad\forall p,q\in \mathbb{N} ^d_0; $$
		\item[(iv)] There exist (resp. for every) constants $h,k>0$ such that
		\begin{equation*} \label{condd}
		\|f (x) e^{h\abs{x}^{1/\tau}}\|_{L^\infty} <+\infty \quad\mbox{and}\quad \| \hat f (\o) e^{k\abs{\o}^{1/\gamma}}\|_{L^\infty} <+\infty;
		\end{equation*}
		\item[(v)] There exist (resp. for every) constants $h,B>0$ such that
		\begin{equation} \label{conde}
		\|(\partial^{q} f)(x)  e^{h \abs{x}^{1/\tau}}\|_{L^\infty} \lesssim B^{|q|} |q| ! ^{\gamma},\quad\forall q\in \mathbb{N} ^d_0.
		\end{equation}
	\end{itemize}
\end{theorem}

Moreover, we could consider any $L^p$-norm, $1\leq p < \infty $ instead of $ L^{\infty}$-norm in Theorem  \ref{simetrija}, cf. \cite{KPP}.

By using Theorem \ref{simetrija} it can be shown that the Fourier transform is a topological isomorphism between
$ \cS_\tau ^\gamma (\rr d) $ and $ \cS_\gamma ^\tau (\rr d) $, $ \gamma, \tau \geq 1/2$ ($ {\mathcal F} (  \cS_\tau ^\gamma ) (\rr d) = \cS_\gamma ^\tau (\rr d)$),
which extends to a continuous linear transform from $ (\cS_\tau ^\gamma )' (\rr d) $ onto $ (\cS_\gamma ^\tau)' (\rr d)$.
Similar considerations hold for partial Fourier transforms with respect to some choice of variables.
In particular, if $ \gamma = \tau $ and $\gamma \geq 1/2 $ then $ {\mathcal F} ( \cS_\gamma ^\gamma)(\rr d)  = \cS_\gamma ^\gamma (\rr d),$
and if moreover $\gamma > 1/2 $, then $ {\mathcal F} ( \Sigma _\gamma ^\gamma)(\rr d)  = \Sigma _\gamma ^\gamma (\rr d),$
and similarly for their distribution spaces. Due to this fact,
corresponding dual spaces are  referred to as \textit{tempered ultra-distributions} (of Roumieu and Beurling  type respectively), see \cite{Pi88}.

The combination of global regularity with suitable decay properties at infinity (cf. \eqref{conde})
which is built in the very definition of $\cS _{\tau} ^{\gamma}(\rr d)$ and $\Sigma _{\tau} ^{\gamma}(\rr d)$,
makes them suitable for the study of different problems in mathematical physics, \cite{GS,Gram, NR}.
We refer to \cite{medit, CordPilRodTeo2010, Nenad2015, Nenad2016} for the study of localization operators in the context of Gelfand-Shilov spaces.
See also \cite{Toft_2012_Bergmann, ToftquasiBanach2017, Toft-ImageBargmann-2017} for related studies.

\subsection{Time-frequency representations}
In this subsection we recall the definitions and basic properties of the short-time Fourier transform and
the (cross-)$\tau$-Wigner distribution.

Given a function $f$ on $\rd$ and $x,\o\in\rd$, the translation operator $T_x$ and the modulation operator $M_\o$ are defined as
\begin{equation*}
 	T_xf(t)\coloneqq f(t-x)\qquad\text{and}\qquad M_\o f(t)\coloneqq e^{2\pi i\o t}f(t)
\end{equation*}
and their composition $\pi\phas\coloneqq M_\o T_x$ is called time-frequency shift. We can now introduce two most commonly used time-frequency representations of a signal $f$, the so-called short-time Fourier transform (STFT) and the (cross-)Wigner distribution.

\begin{definition} \label{def:stft}
Consider a window $g\in\gel(\rd)\smallsetminus\{0\}$. The short-time Fourier transform of $f\in\gel(\rd)$ with respect to $g$ is the function defined on the phase-space as follows:
\begin{equation*}
 		V_gf\phas\coloneqq\<f,\pi\phas g\>=\intrd \overline{g(t-x)}e^{-2\pi it\o}f(t)\,dt,\qquad\phas\in\rdd.
\end{equation*}
\end{definition}

We refer to \cite[Chapter 3]{Grochenig_2001_Foundations} for the properties and different equivalent forms of the STFT.

\begin{definition}  \label{def:tauwigner}
	Let $\tau\in[0,1]$. The (cross-)$\tau$-Wigner distribution of $f,g\in\gel(\rd)$ is defined by
	\begin{equation} \label{tauwigner}
		W_\tau(f,g)\phas\coloneqq\intrd e^{-2\pi it\o}f(x+\tau t)\overline{g(x-(1-\tau)t)}\,dt,\qquad\phas\in\rdd.
	\end{equation}
\end{definition}

When $\tau = 1/2$, $W_{1/2}(f,g)$ is simply called the cross-Wigner distribution of $f$ and $g$ and is denoted by $W(f,g)$ for short.
Both STFT and $W_\tau$ are well defined for $f,g \in L^2(\rd)$ and  if the operator $\cA_\tau$, $\tau\in(0,1)$,
is defined on $ L^2(\rd)$ as
\begin{equation*}
	\cA_\tau f(t)\coloneqq f\left(\frac{\tau-1}{\tau}t\right), \;\;\; t \in \rd,
\end{equation*}
then the connection between the STFT and $\tau$-Wigner distribution is described as follows.

\begin{lemma} \label{lm:wignerstft}
	Let $g\in\gel(\rd)\smallsetminus\{0\}$ and  $f\in\gel(\rd)$.
	\begin{itemize}
		\item[(i)]If $\tau\in(0,1)$, then
		\begin{equation} \label{tauwignerstft}
			W_\tau(f,g)\phas=\frac1{\tau^{d}}e^{2\pi i\frac1{\tau}\om x}V_{\cA_\tau g}f\left(\frac1{1-\tau}x,\frac1{\tau}\om\right), \;\;\;
\forall\,\phas\in\rdd;
		\end{equation}
		\item[(ii)] if $\tau=0$, then
		\begin{equation*}
			W_0(f,g)\phas=e^{-2\pi ix\o}f(x)\overline{\hg(\o)} = R(f,g)\phas, \;\;\;  \forall\,\phas\in\rdd;
		\end{equation*}
		\item[(iii)] if $\tau=1$, then
		\begin{equation*}
		W_1(f,g)\phas=e^{2\pi ix\o}\overline{g(x)}\hf(\o) = \overline{R(g,f) }  \phas, \;\;\; \forall\,\phas\in\rdd;
		\end{equation*}
	\end{itemize}
where $ R(f,g) $ denotes the Rihaczek distribution of $f$ and $g$.
\end{lemma}

\begin{proof}
The proof is straightforward, and we show only {\em (i)} for the sake of completeness (see also \cite[Proposition 1.3.30]{CordRod2020}).
After the change of variables $ s = x+ \tau t $ in \eqref{tauwigner} we obtain
\begin{multline*}
W_\tau(f,g)\phas = \frac{1}{\tau^d} \intrd e^{-2\pi i \frac{1}{\tau}(s-x)\o} f(s)\overline{g(\frac{1}{\tau} (x - (1-\tau)s))}\,ds \\
= \frac{1}{\tau^d} e^{2\pi i \frac{1}{\tau}x\o} \intrd e^{-2\pi i s \frac{\o}{\tau}} f(s)\overline{\cA_\tau g(s -\frac{x}{1-\tau})}\,ds \\
=\frac1{\tau^{d}}e^{2\pi i\frac1{\tau}\om x}V_{\cA_\tau g}f\left(\frac1{1-\tau}x,\frac1{\tau}\om\right), \;\;\;
\forall\,\phas\in\rdd,
\end{multline*}
since $\cA_\tau g (s -\frac{x}{1-\tau}) =  g(\frac{x}{\tau} + \frac{\tau -1}{\tau} s)$.

\par

Notice that when $\tau = 1/2$, we have $\cA_{1/2} g(t) = g(-t),$ and \eqref{tauwignerstft} becomes
$$
W(f,g) \phas  = 2^d e^{4\pi i x\cdot \omega} V_{\cA_{1/2} g} (2x, 2\omega), \;\;\; \forall\,\phas\in\rdd.
$$
\end{proof}

Definitions \ref{def:stft} and \ref{def:tauwigner} are uniquely extended  to $f\in(\gel ) '(\rd)$ by duality.

We will also use the following fact related to time-frequency representations of the Gelfand-Shilov spaces.

\begin{theorem}\label{potreban}
Let  $\cS ^{*}(\rr d) $ denote $\cS ^{\{\gamma\}}(\rr d)  $, $ \gamma \geq 1/2$,
or $\cS ^{(\gamma)}(\rr d)$, $ \gamma > 1/2$. Moreover, let $g\in\cS^{*}(\rd) \smallsetminus \{ 0\}$ and $\tau\in[0,1]$. Then the following are true:
	\begin{itemize}
		\item[(i)] if $f\in \cS^{*} (\rd)$, then $W_\tau(f,g),V_gf\in\cS^{*}(\rdd)$;
		\item[(ii)] if $ f \in  (\cS ^{*})' (\rd)$ and $W_\tau(f,g)\in\cS^{*} (\rdd)$ or $V_gf\in\cS^{*} (\rdd)$, then $f \in \cS^{*} (\rd)$.
	\end{itemize}
\end{theorem}

\begin{proof}
The proof for the STFT and $W_{1/2}$ can be found in several sources, see e.g. \cite{GZ,Te1, Toft_2012_Bergmann}.
The case $\tau \in [0,1], $ $\tau \neq 1/2$ can be proved in a similar fashion and is left for the reader as an exercise.
\end{proof}

\subsection{Pseudodifferential and localization operators} \label{subs:psido}
Next we introduce $\tau$-quan\-ti\-za\-ti\-ons as pseudodifferential operators acting on $ \gel(\rd)$.
We address the reader to the textbooks \cite{CordRod2020,Grochenig_2001_Foundations} in which the framework is mostly the one of $\cS (\rd)$ and $\cS' (\rd)$,
and we suggest \cite{NR, Te1, Nenad2015, Toft_2012_Bergmann, ToftquasiBanach2017, Toft-ImageBargmann-2017} for the framework of Gelfand-Shilov spaces and their spaces of ultra-distributions.

\begin{definition} \label{def:tauop}
Let $\tau\in[0,1]$. Given a symbol $\sigma\in\ult(\rdd)$, the $\tau$-quantization of $\sigma$
is the pseudodifferential operator
	\begin{equation*}
		\Opt(\sigma)\colon\gel(\rd)\to\ult(\rd)
	\end{equation*}
	defined by the formal integral
	\begin{equation} \label{eq:taupsido}
	\Opt(\sigma)f(x)\coloneqq\iintrdd e^{2\pi i(x-y)\om}\sigma\left((1-\tau)x+\tau y,\om\right)f(y)\,dyd\om,
	\end{equation}
	or, in a weak sense,
	\begin{equation*}
		\<\Opt(\sigma)f,g\> =
\int_{\rd} \iintrdd e^{2\pi i(x-y)\om}\sigma\left((1-\tau)x+\tau y,\om\right)f(y) \overline{g}(x)\,dyd\om dx,
	\end{equation*}
$ f,g\in\gel(\rd).$
\end{definition}

The correspondence between the symbol $\sigma $ and the operator $\Opt(\sigma)$ given by \eqref{eq:taupsido} is known as the Shubin $\tau$-representation, \cite{Shubin91}. By a change of variables and an interchange of the order of integration, it can be shown that
$\Opt(\sigma)$, $\sigma \in \ult(\rdd)$, and the (cross-)$\tau$-Wigner distribution are related by the following formula:
	\begin{equation} \label{optauwigner}
		\<\Opt(\sigma)f,g\> =
\<\sigma,W_\tau(g,f)\>,\qquad f,g\in\gel(\rd).
	\end{equation}
Thus, for  $\tau = 1/2 $ (the  Weyl quantization) we recover the Weyl pseudodifferential operators, and when $\tau = 0 $ we obtain the Kohn-Nirenberg operators. Commonly used equivalent notation for the Weyl operators  in the literature
are $\OpW(\sigma)$, $\Opw(\sigma)$, $L_\sigma$ or $\sigma^w$. The Weyl calculus reveals to be extremely important since every continuous and linear operator from $\gel(\rd)$ into $\ult(\rd)$ can be written as the Weyl transform of some (Weyl) symbol $\sigma\in\ult(\rdd)$. This is due to the Schwartz kernel theorem when extended to the duality between $\gel(\rd)$ and $\ult(\rd)$, see \cite{LCPT, Nenad2015}.

\par

Next we introduce localization operators in the form of the STFT multipliers, and discuss their relation to $\tau$-quantizations given above.

\begin{definition} \label{def:locop}
	Consider windows $\f_1,\f_2\in\gel(\rd)\smallsetminus\{0\}$ and a symbol $a\in\ult(\rdd)$. The localization operator
	\begin{equation*}
		\gaw\colon\gel(\rd)\to\ult(\rd)
	\end{equation*}
	is the continuous and linear mapping formally defined by
	\begin{equation*}
		\gaw f(t)\coloneqq\iintrdd a\phas V_{\f_1}f\phas M_\o T_x \f_2(t)\,dxd\o,
	\end{equation*}
	or, in a weak sense,
	\begin{equation} \label{eq:locopweak}
		\<\gaw f,g\>\coloneqq\<a,\overline{V_{\f_1}f} V_{\f_2}g\>,\qquad f,g\in\gel(\rd).
	\end{equation}
\end{definition}

It can be proved that every localization operator $\gaw$ can be written in the Weyl form, i.e. identified with
the Weyl pseudodifferential operator due to the following formula
\begin{equation} \label{eq:locopweyl}
	\gaw=\mathrm{Op}_{1/2} (a\ast W(\f_2,\f_1)),
\end{equation}
and $\sigma=a\ast W(\f_2,\f_1)$ is called Weyl symbol of $\gaw$.  We refer to \cite[Lemma 2.4]{BCG02} or \cite{folland89} for the proof, see also \cite{Nenad2016}.

By combining \eqref{optauwigner} and \eqref{eq:locopweyl}
we define $\tau$-localization operators as follows.

Let there be given $\tau \in [0,1]$, windows $\f_1,\f_2\in\gel(\rd)\smallsetminus\{0\}$ and a symbol $a\in\ult(\rdd)$. Then  $\tau$-localization operator
is defined to be
\begin{equation} \label{eq:locopweyltau}
	\gawtau \coloneqq\mathrm{Op}_{\tau} (a\ast W_{\tau} (\f_2,\f_1)).
\end{equation}
In other words, every $\tau$-localization operator is identified with $\tau$-pseudodifferential operator associated to the symbol
$ \sigma_\tau = a\ast W_{\tau} (\f_2,\f_1).$

However, it turns out that the class of localization operators given by \eqref{eq:locopweyltau} coincides to the one given by Definition \ref{def:locop}, see  \cite{Toft-2013}. We give an independent proof based on the kernel argument.

\begin{proposition}  \label{pr:taulocop}
Let $\f_1,\f_2\in\gel(\rd)\smallsetminus\{0\}$, $a\in\ult(\rdd)$ and $\tau \in [0,1].$ Then
$$
\gaw = \gawtau.
$$
\end{proposition}

\begin{proof}
By the Schwartz kernel theorem for $\gel(\rd)$ and $\ult(\rd)$, it suffices to show that the kernels of $ \gaw $ and $ \gawtau$ coincide.
From \eqref{eq:locopweak} it follows that
\begin{multline*}
\langle \gaw f, g \rangle
\\[1ex]
= \iint_{\mathbb{R}^{2d}} a(x,\omega)
\left ( \int_{\mathbb{R}^{d}} f(y) \overline{M_\omega T_x \varphi _1} (y) dy \right )
\left ( \int_{\mathbb{R}^{d}} \overline{g} (t) M_\omega T_x \varphi _2 (t) dt \right ) dx d\omega
\\[1ex]
= \int_{\mathbb{R}^{d}} \int_{\mathbb{R}^{d}} f(y) \overline{g} (t)
\left ( \iint_{\mathbb{R}^{2d}} a(x,\omega) \overline{M_\omega T_x \varphi _1} (y)
 M_\omega T_x \varphi _2 (t) dx d\omega \right )   dt  dy
=
\langle k, g \otimes \overline{f} \rangle,
\end{multline*}
so the kernel of $ \gaw $ is given by
\begin{equation} \label{kernel}
k(t,y) =  \iintrdd a(x,\omega) \overline{M_\omega T_x \varphi _1} (y)
M_\omega T_x \varphi _2 (t) dx d\omega.
\end{equation}

It remains to calculate the kernel of  $ \gawtau$.
We first calculate $a\ast W_{\tau} (\f_2,\f_1)$:
\begin{multline*}
a \ast W_{\tau} (\varphi_2,\varphi_1)(p,q)
=
\iint_{\mathbb{R}^{2d}}  a(x,\omega) W_{\tau} (\varphi_2,\varphi_1) (p-x,q-\omega) dx d\omega
\\[1ex]
=
\iint_{\mathbb{R}^{2d}}  a(x,\omega) W_{\tau} (T_x M_\omega \varphi_2, T_x M_\omega \varphi_1) (p,q) dx d\omega
\\[1ex]
=
\iint_{\mathbb{R}^{2d}}  a(x,\omega)
\left ( \int_{\mathbb{R}^{d}} T_x M_\omega \varphi_2(p+ \tau s )\overline{ T_x M_\omega \varphi_1}
(p- (1-\tau)s)e^{-2\pi i q s} ds \right) dx d\omega
\\[1ex]
=
\iint_{\mathbb{R}^{2d}}  a(x,\omega)
\left ( \int_{\mathbb{R}^{d}}  M_\omega T_x \varphi_2(p+ \tau s )\overline{ M_\omega T_x \varphi_1}
(p- (1-\tau)s)e^{-2\pi i q s} ds \right) dx d\omega,
\end{multline*}
where we have used the commutation relation $ T_x M_\omega = e^{-2 \pi i x \omega} M_\omega T_x ,$
and the covariance property of $\tau$-Wigner transform:
$$
W_{\tau}(T_x M_\omega f, T_x M_\omega g)  (p,q) = W_{\tau} (f,g) (p-x, q-\omega),
$$
which follows by direct calculation.

\par

Now we have
\begin{multline*}
\langle  \mathrm{Op}_{\tau} (a\ast W_{\tau} (\f_2,\f_1)) f, g \rangle =
\langle  a\ast W_\tau (\varphi_2,\varphi_1), W_\tau ( g, f) \rangle
\\[1ex]
=
\iint_{\mathbb{R}^{2d}}   a(x,\omega)
\iint_{\mathbb{R}^{2d}} \big (
\iint_{\mathbb{R}^{2d}}  M_\omega T_x \varphi_2(p+ \tau s )\overline{ M_\omega T_x \varphi_1}
(p-(1-\tau)s)e^{-2\pi i q (s-r)}
\\[1ex]
\times
\overline{g}(p+ \tau r) f (p- (1-\tau)r )
ds dr \big ) dp dq dx d\omega
\\[1ex]
=
\iint_{\mathbb{R}^{2d}}   a(x,\omega)
\int_{\mathbb{R}^{d}} \big (
\iint_{\mathbb{R}^{2d}}  M_\omega T_x \varphi_2(p+ \tau s )\overline{ M_\omega T_x \varphi_1}
(p-(1-\tau)s)
\\[1ex]
\times \overline{g}(p+ \tau r) f (p-(1-\tau)r ) \delta (r-s) ds
dr \big ) dp dx d\omega
\\[1ex]
=
\iint_{\mathbb{R}^{2d}}   a(x,\omega)
\int_{\mathbb{R}^{d}} \big (
\int_{\mathbb{R}^{d}}  M_\omega T_x \varphi_2(p+ \tau s )\overline{ M_\omega T_x \varphi_1}
(p-(1-\tau)s)
\\[1ex]
\times
\overline{g}(p+ \tau s) f (p-(1-\tau)s))
ds \big ) dp dx d\omega.
\end{multline*}
where we used a suitable interpretation of the oscillatory integrals in the distributional sense.
In particular, the Fourier inversion formula in the sense of distributions gives
$ \int e^{2\pi i x \omega} d \omega = \delta (x),$
where $\delta$ denotes the Dirac delta, whence
$$ \iint_{\rdd} \phi (x)  e^{2\pi i (x-y) \omega} d x d \omega = \phi (y), \;\;\;
\phi \in {\mathcal S}^{( 1)} (\mathbb{R}^{d}).
$$

Finally, the change of variable $ p+ \tau s = t$ and  $ p- (1-\tau)s = y$
gives
\begin{multline*}
\langle  \mathrm{Op}_{\tau} (a\ast W_{\tau} (\f_2,\f_1)) f, g \rangle  \\
= \iint_{\mathbb{R}^{2d}} \iint_{\mathbb{R}^{2d}}   a(x,\omega) M_\omega T_x \varphi_2(t) \overline{ M_\omega T_x \varphi_1} (y) dx d\omega
 \overline{g}(t) f (y)  dt dy \\
= \langle k_\tau, g \otimes \overline{f} \rangle = \langle k, g \otimes \overline{f} \rangle,
\end{multline*}
where $k$ is given by \eqref{kernel}. By the uniqueness of the kernel we conclude that
$$
\gaw =  \gawtau
$$
and the proof is finished.
\end{proof}
	
\subsection{Ultra-modulation spaces} \label{subs:modsp}
We use the terminology ultra-modulation spaces in order to emphasize that such spaces may contain ultra-distributions, contrary to the most usual situation when members of modulation spaces are tempered distributions. However, ultra-modulation spaces belong to the  family
of modulation spaces introduced in \cite{feichtinger-modulation}. We refer to e.g. \cite{Toft-GaborAnalysis-2015,Toft-ImageBargmann-2017} for a  general approach to the broad class of modulation spaces.
	
\par

\begin{definition}\label{Def-UltraMod}
	Fix a non-zero window $g\in\gel(\rd)$, a weight $m\in\mathscr{P}_E(\rdd)$ and $0<p,q\leq \infty$. The ultra-modulation space $M^{p,q}_m(\rd)$ consists of all tempered ultra-distributions $f\in\ult(\rd)$ such that the (quasi-)norm
	\begin{equation}\label{Eq-NormUltraMod}
	\|f\|_{M^{p,q}_m}\coloneqq\|V_gf\|_{L^{p,q}_m}=\left(\intrd\left(\intrd |V_g f \phas|^p m\phas^p dx  \right)^{\frac qp}d\o\right)^\frac1q
	\end{equation}
	(obvious modifications with $p=\infty$ or $q=\infty)$ is finite.
\end{definition}

We write $M^p_m(\rd)$ for $M^{p,p}_m(\rd)$, and $M^{p,q}(\rd)$ if $m\equiv 1$.

We recall that the spaces $M^{p,q}_m(\rd)\subset\cS'(\rd)$, with $1\leq p,q\leq \infty$, $g\in\cS(\rd)$ and $m$ of at most polynomial growth at infinity, were invented by H. Feichtinger in \cite{feichtinger-modulation} and called \textit{modulation spaces}. There it was proved that they are Banach spaces and that different window functions in $\cS(\rd) \smallsetminus \{ 0\}$
yield equivalent norms. Moreover, the window class can be enlarged to  the Feichtinger algebra $M^{1,1}_v(\rd)$, where $v$ is a submultiplicative weight of at most polynomial growth at infinity such that $m$ is $v$-moderate.

It turned out that properties analogous to the Banach case hold in the quasi-Banach one as well, see \cite{Galperin2004}. Moreover, such properties remain valid also in the more general setting of Definition \ref{Def-UltraMod}. We collect them in the following theorem in the same manner of \cite{Toft-GaborAnalysis-2015,ToftquasiBanach2017}, see references therein also.

\begin{theorem}\label{Th-PoertyUltraMod}
Consider  $0<p,p_1,p_2,q,q_1,q_2\leq\infty$ and weights $m,m_1,m_2\in\mathscr{P}_E(\rdd)$. Let $\norm{\cdot}_{M^{p,q}_m}$ be given by \eqref{Eq-NormUltraMod} for a fixed $g\in\gel(\rd)\smallsetminus\{0\}$. Then:
	\begin{itemize}
		\item[(i)] $\left(M^{p,q}_m(\rd),\norm{\cdot}_{M^{p,q}_m}\right)$ is a quasi-Banach space whenever at least one between $p$ and $q$ is strictly smaller than $1$, otherwise it is a Banach space;
		\item[(ii)] if $\tilde g \in\gel(\rd)\smallsetminus\{0\}$, $\tilde g\neq g$, then it induces a (quasi-)norm equivalent to $\norm{\cdot}_{M^{p,q}_m}$;
		\item[(iii)] if $p_1\leq p_2$, $q_1\leq q_2$ and $m_2\lesssim m_1$, then:
		\begin{equation*}
			\gel(\rd)\hookrightarrow M^{p_1,q_1}_{m_1}(\rd)\hookrightarrow M^{p_2,q_2}_{m_2}(\rd) \hookrightarrow (\gel)'(\rd),
		\end{equation*}
and the inclusions are dense;
		\item[(iv)] if $p,q<\infty$, then :
		\begin{equation*}
			\left(M^{p,q}_m(\rd)\right)'\cong M^{p',q'}_{1/m}(\rd),
		\end{equation*}
		where
		\begin{equation*}
			p'\coloneqq
			\begin{cases}
			\infty\quad&\text{if}\quad0<p\leq1\\
			\frac{p}{p-1}\quad&\text{if}\quad1<p<\infty
			\end{cases}
		\end{equation*}
		and similarly for $q'$.
	\end{itemize}
\end{theorem}

\begin{remark}
Point \textit{(ii)} of the previous theorem tell us that the definition of $M^{p,q}_m(\rd)$ is independent of the choice of the window. Moreover, it can be shown that the class for window functions can be extended from $\gel(\rd)$ to $M^r_v(\rd)$, where $r\leq p,q$ and $v\in\mathscr{P}_E(\rdd)$ is  submultiplicative and such that $m$ is $v$-moderate, \cite{ToftquasiBanach2017}.

We refer to \cite{Cordero-GelfandShilov-2007} for the density of $\gel(\rd)$ in $M^{p,q}_m(\rd)$.
\end{remark}

The following proposition is proved in e.g. \cite[Theorem 4.1]{T3}, \cite[Theorem 3.9]{Toft_2012_Bergmann}.

\begin{proposition}
	Consider $\ga\geq1$ and $0<p,q\leq\infty$. Then
	\begin{equation*}
		\cS^{(\ga)}(\rd)=\bigcap_{k\geq0}M^{p,q}_{w^\ga_k}(\rd),\qquad \cS^{(\ga)'}(\rd)=\bigcup_{k\geq0}M^{p,q}_{1/w^\ga_k}(\rd).
	\end{equation*}
\end{proposition}

In some situations it is convenient to consider (ultra)-modulation spaces as subspaces of $\cS^{\{1/2\}'}(\rd)$ (taking the window $g$ in $\cS^{\{1/2\}}(\rd)$), see for example \cite{Cordero-GelfandShilov-2007,ToftquasiBanach2017}. However, for our purposes it is sufficient to
consider the weights in $\mathscr{P}_E(\rdd)$, and then $M^{p,q}_m(\rd)$ is  a subspace of $\ult(\rd)$. We address the reader to \cite[Proposition 1.1]{ToftquasiBanach2017} and references quoted there for more details.

\par

We restate \cite[Proposition 2.6]{medit} in a simplified case suitable to our purposes.

\begin{proposition} \label{prop:inversion}
	Assume $1\leq p,q\leq\infty$, $m\in\mathscr{P}_E(\rdd)$ and $g\in\gel(\rd)$ such that $\norm{g}_2=1$. Then for every $f\in M^{p,q}_m(\rd)$ the following inversion formula holds true:
	\begin{equation}\label{invformula}
		f=\iintrdd V_gf\phas M_\o T_xg\,dxd\o,
	\end{equation}
	where the equality holds in $M^{p,q}_m(\rd)$.
\end{proposition}

The embeddings between modulation spaces are studied by many authors. We recall a recent contribution
\cite[Theorem 4.11]{Guo2019}, which is convenient for our purposes, and which will be used in Lemma \ref{Lem-inclusion-for-compactness}.

\begin{theorem}\label{Th-GuoChenFanZhao}
	Let $0<p_j,q_j\leq\infty$, $s_j,t_j\in\bR$ for $j=1,2$ and consider the polynomial weights $v_{t_j}$, $v_{s_j}$ defined as in \eqref{Eq-DefPolynomialWeights}. Then
	\begin{equation*}
		M^{p_1,q_1}_{v_{t_1}\otimes v_{s_1}}(\rd)\hookrightarrow M^{p_2,q_2}_{v_{t_2}\otimes v_{s_2}}(\rd)
	\end{equation*}
	if  the following two conditions hold true:
	\begin{itemize}
		\item[(i)] $(p_1,p_2,t_1,t_2)$ satisfies one of the following conditions:
		\begin{align*}
			(\cC_1)&\qquad\frac1{p_2}\leq\frac1{p_1},\qquad t_2\leq t_1,\\
			(\cC_2)&\qquad\frac1{p_2}>\frac1{p_1},\qquad\frac1{p_2}+\frac{t_2}{d}<\frac1{p_1}+\frac{t_1}{d};
		\end{align*}
		\item[(ii)] $(q_1,q_2,s_1,s_2)$ satisfies one  of the conditions $(\cC_1)$ or $(\cC_2)$ with $p_j$ and $t_j$ replaced by $q_j$ and $s_j$ respectively.
	\end{itemize}
\end{theorem}


\subsection{Gabor Frames}
Consider a lattice $\Lambda\coloneqq\alpha\zd\times \beta\zd\subset\rdd$ for some $\al,\be>0$.
Given $g\in L^2(\rd)\smallsetminus\{0\}$, the set  of time-frequency shifts $\G(g,\Lambda)\coloneqq\{\pi(\lambda)g:\
\lambda\in\Lambda\}$ is called a
\textit{Gabor system}. The set $\G(g,\Lambda)$ is
a \textit{Gabor frame} if there exist constants $A,B>0$ such that
\begin{equation}\label{gaborframe}
A\|f\|_{L^2}^2\leq\sum_{\lambda\in\Lambda}|\langle f,\pi(\lambda)g\rangle|^2\leq B\|f\|^2_{L^2},\qquad \forall f\in L^2(\rd).
\end{equation}
If $\G(g,\Lambda)$  is a Gabor frame, then the \textit{frame operator}
\begin{equation*}
	Sf\coloneqq\sum_{\lambda\in\Lambda}\langle f,\pi(\lambda)g\rangle\pi(\lambda)g,\quad f\in\lrd,
\end{equation*}
is a topological isomorphism on $\lrd$. Moreover, if we define $h\coloneqq S^{-1}g\in\lrd$, then the system $\G(h,\Lambda)$ is a Gabor frame and we have the reproducing formulae
\begin{equation}\label{RF}
f=\sum_{\lambda\in\Lambda}\langle f,\pi(\lambda)g\rangle\pi(\lambda)h=\sum_{\lambda\in\Lambda}\langle f,\pi(\lambda)h\rangle \pi(\lambda)g,\quad \forall,f\in\lrd,
\end{equation}
with unconditional convergence in $L^2(\rd)$. The window $h$ is called the \textit{canonical dual window} of  $g$. In particular, if $h=g$ and $\|g\|_{L^2}=1$ then $A=B=1$, the frame operator is the identity on $L^2(\rd)$ and the Gabor frame is called \textit{Parseval} Gabor frame. In particular, from \eqref{gaborframe} we can recover exactly the $L^2$-norm of every vector:
\begin{equation*}
	\|f\|_{L^2}^2=\sum_{\lambda\in\Lambda}|\langle f,\pi(\lambda)g\rangle|^2,\quad \forall\,f\in L^2(\rd).
\end{equation*}
 Any window $u\in\lrd$ such that \eqref{RF} is satisfied is called \textit{alternative dual window} for $g$.
 Given two functions $g,h\in\lrd$ we are able to extend the notion of Gabor frame operator to the operator $S_{g,h}=S^{\Lambda}_{g,h}$ in the following way:
 \begin{equation*}
 	S_{g,h} f\coloneqq\sum_{\lambda\in\Lambda}\langle f,\pi(\lambda)g\rangle\pi(\lambda)h,\quad f\in\lrd,
 \end{equation*}
 whenever this is well defined. With this notation the reproducing formulae \eqref{RF} can be rephrased as  $S_{g,h}=I=S_{h,g}$, where $I$ is the identity on $\lrd$.

Discrete equivalent norms produced by means of Gabor frames make of ultra-modulation spaces a natural framework for time-frequency analysis.  We address the reader to  \cite{Galperin2004,Grochenig_2001_Foundations,Toft-GaborAnalysis-2015,ToftquasiBanach2017}.
\begin{theorem}\label{framesmod}
Consider $m,v\in\mathscr{P}_E(\rdd)$ such that $v$ is submultiplicative and $m$ is $v$-moderate. Take $\Lambda\coloneqq\a \zd\times \b \zd$, for some $\al,\be>0$,  and $g,h\in \gel(\rd)$ such that $S_{g,h}=I$ on $\lrd$. Then
\begin{equation*}\label{RFmpq}
f=\sum_{\lambda\in\Lambda}\langle f,\pi(\lambda)g\rangle\pi(\lambda)h=\sum_{\lambda\in\Lambda}\langle f,\pi(\lambda)h\rangle \pi(\lambda)g,\quad\forall\,f\in M^{p,q}_m(\rd),
\end{equation*}
with unconditional convergence in $M^{p,q}_m(\rd)$ if $0<p,q<\infty$ and with weak-* convergence in $M^\infty_{1/v}(\rd)$ otherwise. Moreover, there exist $0<A\leq B$ such that, for every $f\in M^{p,q}_m(\rd)$,
 \begin{equation*}\label{gaborframe2}
 A\|f\|_{M^{p,q}_m}\leq\left(\sum_{n\in\zd}\left(\sum_{k\in\zd}|\langle f,\pi(\a k,\beta n)g\rangle|^p m(\a k, \beta n)^p\right)^{\frac qp}\right)^{\frac1q}\leq B\|f\|_{M^{p,q}_m},
 \end{equation*}
 independently of $p,q$, and $m$. Equivalently:
 \begin{equation}\label{Eq-EquivalenceContDiscNormUltraMod}
 \|f\|_{M^{p,q}_m(\rd)}\asymp \|(\la f ,\pi(\lambda)g\ra)_\lambda \|_{\ell^{p,q}_{m}(\Lambda)}= \|(V_g f(\lambda))_\lambda\|_{\ell^{p,q}_{m}(\Lambda)}.
 \end{equation}
 Similar inequalities hold with $g$ replaced by $h$.
\end{theorem}


Now we are able to prove the  convolution relations for ultra-modulations spaces which will be used to prove our main results in Section
\ref{main results}. For the Banach cases with weight of at most polynomial growth at infinity, convolution relations were studied in e.g \cite{EleCharly2003,toft1,Toftweight2004}. We modify the technique used in \cite{BasCorNic2019} to the Gelfand-Shilov framework presented so far. The essential tool is the equivalence between continuous and discrete norm \eqref{Eq-EquivalenceContDiscNormUltraMod}.

\begin{proposition}\label{Pro-ConvRel}
	Let there be given $0<	p,q,r,t,u,\gamma\leq\infty$ such that
	\begin{equation*}\label{Holderindices}
	\frac 1u+\frac 1t=\frac 1\gamma,
	\end{equation*}
	and
	\begin{equation*}\label{Youngindicesrbig}
	\frac1p+\frac1q=1+\frac1r,\quad \,\, \text{ for } \, 1\leq r\leq \infty
	\end{equation*}
	whereas
	\begin{equation*}\label{Youngindicesrbig}
	p=q=r,\quad \,\, \text{ for } \, 0<r<1.
	\end{equation*}
	Consider  $m,v, \nu \in\mathscr{P}_E(\rdd)$ such that $m$ is $v$-moderate. Then
	\begin{equation*}\label{mconvm}
	M^{p,u}_{m|_1\otimes \nu}(\Ren)\ast  M^{q,t}_{v|_1\otimes
		v|_2\nu^{-1}}(\Ren)\hookrightarrow M^{r,\gamma}_m(\Ren),
	\end{equation*}
	where $m|_1$, $v|_1$, $v|_2$ are defined as in \eqref{Eq-Restrictions}.
\end{proposition}

\begin{proof}
First observe that due to Lemma \ref{Lem-Weight1} and Lemma \ref{Lem-Weight2} it follows that the ultra-modulation spaces which came into play are well defined.

The main tool is the idea contained in \cite[Proposition 2.4]{EleCharly2003}. We take the ultra-modulation norm with respect to the Gaussian windows	$g_0(x) \coloneqq e^{-\pi{x^2}}\in\cS^{\{1/2\}}(\rd)$ and
	$g(x)\coloneqq2^{-d/2}e^{-\pi x^2/2} = (g_0\ast g_0)(x)\in\cS^{\{1/2\}}(\rd)$.

	Since the involution operator $g^\ast(x)= \overline{g(-x)}$ and the modulation operator $M_\o$ commute, by a direct computation we have
	\begin{equation*}
	M_\o(g_0^*\ast
	g_0^*)=M_\o g_0^*\ast  M_\o g_0^*
	\end{equation*}
	and
	\begin{equation*}
		V_g f\phas =e^{-2\pi i x\o}(f\ast M_\o g^*)(x).
	\end{equation*}
  Thus, by using the associativity and commutativity of the convolution product, we obtain
  \begin{equation*}
  	V_g(f\ast h)\phas =e^{-2\pi i x\o}\big( (f\ast h)\ast M_\o g^*\big)(x)
  	=e^{-2\pi i x\o}\big( (f\ast M_\o g_0^*) \ast (h\ast M_\omega g_0^*)
  	\big)(x) \, .
  \end{equation*}
	We use the norm equivalence \eqref{Eq-EquivalenceContDiscNormUltraMod} for a suitable $\Lambda=\a\zd\times\b\zd$, and then  the $v$-moderateness in order to majorize $m$:
	\begin{equation*}
		m(\a k, \b n) \lesssim \linebreak[3] m(\a k,0)v(0,\b n ) = m|_1(\a k) v|_2(\b n).
	\end{equation*}

	Eventually Young's convolution 	inequality for sequences is used in the $k$-variable  and
	H\"older's one in the $n $-variable. Indeed both inequalities can be used since $p,q,r,\gamma,t,u$
fulfill the assumptions of the proposition. We write in details the case when $r,\gamma,t,u<\infty$,
and leave to the reader  the remaining cases, when one among the indices $r,\gamma,t,u$ is equal to $\infty$, which can be done analogously.

	\begin{multline*}
		\|f\ast h\|_{M^{r,\gamma}_m} \asymp  \|((V_{g}(f\ast h))(\a k,\b n)m(\a k, \b n))_{k,n}\|_{\ell^{r,\gamma}(\zdd)}\\
\lesssim \left( \sum_{n\in\zd} \left( \sum_ {k\in\zd} |(f\ast M_{\b n} g_0^*) \ast (h\ast
		M_{\b n} g_0^*) (\a k)|^r m|_1(\a k) ^r \,\right) ^{\gamma/r} \, v|_2(\b n) ^\gamma
	 \right)^{1/\gamma} \\
= \left( \sum_{n\in\zd} \| (f\ast M_{\b n} g_0^*) \ast (h\ast
		M_{\b n} g_0^*) \|^\gamma_{\ell^{r}_{ m|_1}(\a\zd)} v|_2(\b n)^\gamma
		\right)^{1/\gamma}\\
\lesssim  \left(\sum_{n\in\zd}\|f\ast M_{\b n} g_0^*\|_{\ell^{p}_{m|_1}(\a\zd)}^\gamma\,
		\|h\ast M_{\b n} g_0^*\|_{\ell^{q}_{v|_1}(\a\zd)}^\gamma\, v|_2(\b n)^\gamma\,\right)^{1/\gamma}\\
\lesssim  \left(\sum_{n\in\zd} \!\!\|f\ast M_{\b n} g_0^*\|_{\ell^{p}_{m|_1}(\a\zd)}^{u} \nu
		(\b n )^{u}   \right) ^{\frac{1}{u}} \!
		\left( \sum_{n\in\zd} \!\!\|h\ast M_{\b n} g_0^*\|_{\ell^{q}_{v|_1}(\a\zd)}^{t}\frac{v|
			_2(\b n)^{t}}{\nu(\b n)^{t}}\! \right) ^{\frac{1}{t}} \\
= \|((V_{g_0}f)(\lambda))_{\lambda}\|_{\ell^{p,u}_{m|_1\otimes\nu}(\Lambda)} \,\,		 \|((V_{g_0}h)(\lambda))_{\lambda}\|_{\ell^{q,t}_{v|_1\otimes v|_2\nu \inv}(\Lambda)}\\
\asymp  \|f\|_{M^{p,u}_{m|_1\otimes \nu }} \,\,
		\|h\|_{M^{q,t}_{v|_1\otimes v|_2 \nu \inv }}.
	\end{multline*}
This concludes the proof.
\end{proof}

\section{Main Results} \label{main results}

An important relation between the action of an operator $\Opt(\sigma)$ on time-frequency shifts and the STFT of its symbol $\sigma$
is explained in \cite{CordNicTra2019}.
The setting given there is the one of $\cS(\rd)$ and $\cS'(\rd)$, but it is easy to see that the claim is still valid when dealing with $\gel(\rd)$ and $\ult(\rd)$. Moreover, $\gel(\rd)$ and its dual can be replaced by $\cS^{\{\ga\}}(\rd)$ and $\cS^{\{\ga\}'}(\rd)$ as it is done in \cite{CorNicRod-2014} when $\tau=1/2$.
Thus, the proof of the following lemma is omitted, since it follows by a slight modification of the proof of \cite[Lemma 4.1]{CordNicTra2019}.

\begin{lemma}\label{lemma41} Consider $\tau\in[0,1]$, $g\in \gel(\rd)$, $\Phi_\tau\coloneqq W_\tau(g,g)\in\gel(\rdd)$. If $\sigma\in \ult(\rdd)$, then
	\begin{equation}\label{311}
	|\la \Opt(\sigma)\pi(z)g,\pi(w) g\ra|=\left|V_{\Phi_\tau} \sigma\left(\cT_\tau(w,z),J(w-z)\right)\right|, \quad\forall\,z,w\in\rdd,
	\end{equation}
	where $z=(z_1,z_2),w=(w_1,w_2)\in\rdd$ and $\cT_\tau$ and $J$ are defined as follows:
	\begin{equation}\label{Eq-def-T-J}
		\cT_\tau(w,z)\coloneqq\left((1-\tau)w_1+\tau z_1,\tau w_2+(1-\tau)z_2\right),\qquad J(z)\coloneqq(z_2,-z_1).
	\end{equation}
\end{lemma}

The following lemma can be viewed as a form of the inversion formula \eqref{invformula}. The independent proof is given in the Appendix.

\begin{lemma} \label{lm:optaustft}
Let $\tau\in[0,1]$ and $\sigma \in \ult(\R^{2d})$. If
$g \in \gel(\rd)$ with $\norm{g}_{L^2}=1$ and $f\in\gel(\rd)$, then
\begin{equation}\label{Eq-Proof-ii}
\Opt(\sigma) f=\intrdd V_gf(z)\Opt(\sigma)(\pi(z)g)\,dz,
\end{equation}
in the sense that
\begin{equation*}
\<\Opt(\sigma) f,\f\>=\intrdd V_gf(z)\,\<\Opt(\sigma)(\pi(z)g),\f\>\,dz,\qquad\forall\,\f\in\gel(\rd).
\end{equation*}
\end{lemma}

Next we show how the $\tau$-quantization $\Opt(\sigma)$, $\tau\in[0,1]$, can be extended between ultra-modulation spaces under suitable assumptions on the weights. We remark that the following theorem is contained in the more general \cite[Theorem 3.1]{ToftquasiBanach2017}.
Nevertheless our more elementary proof is independent and self-contained. We note that
\cite[Theorem  3.3]{BasCorNic2019} is a particular case of  Theorem \ref{Th-ExtensionWeylOp} when restricted to polynomial weights and
the duality between $\sch(\Ren)$  and  $\sch'(\Ren)$.

\begin{theorem}\label{Th-ExtensionWeylOp}
Consider $\tau\in[0,1]$, $m_0\in\mathscr{P}_E(\bR^{4d})$ and $m_1,m_2\in\mathscr{P}_E(\rdd)$ such that
\begin{equation}\label{Eq-ConditionWeights}
	\frac{m_2\phas}{m_1(y,\eta)}\lesssim m_0((1-\tau)x+\tau y,\tau\o+(1-\tau)\eta,\o-\eta,y-x),\quad\forall\,x,\o,y,\eta\in\rd.
\end{equation}
Fix a symbol $\sigma \in M^{\infty,1}_{m_0}(\R^{2d})$. Then the
pseudodifferential operator $\Opt(\sigma)$, from $\gel(\rd)$ to $\ult(\rd)$,  extends uniquely to a bounded and linear operator from
$M^p_{m_1}(\rd)$ into $M^p_{m_2}(\rd)$ for every $1\leq p<\infty$.
\end{theorem}

\begin{proof}
Let $g \in \gel(\rd)$ with $\norm{g}_{L^2}=1$ and consider $f\in\gel(\rd)\subset M^p_{m_1}(\rd)$.
Due to the normalization chosen $\norm{g}_{L^2}=\norm{\hat{g}}_{L^2}$ and we recall the inversion formula \eqref{invformula} which can be seen as a pointwise equality between smooth functions in this case (see \cite[Proposition 11.2.4]{Grochenig_2001_Foundations}): $f=\int_{\rdd} V_gf(z)\pi(z)g\,dz$.

Next we use Lemma \ref{lm:optaustft} and express the STFT of the tempered ultra-distribution $\Opt(\sigma) f$, with $\tau\in[0,1]$, in the following way:
\begin{equation}\label{Eq-Proof-ii-2}
V_g (\Opt(\sigma) f)(w)=\<\Opt(\sigma) f,\pi(w)g\>\overset{\text{\eqref{Eq-Proof-ii}}}{=}\int_{\rdd} V_g f(z) \, \langle \Opt(\sigma) \pi(z) g,
\pi(w)g\rangle \,dz.
\end{equation}

In the next step we prove that the map $ M_\tau(\sigma): G\mapsto M_\tau(\sigma)G $, defined by
\begin{equation*}
M_\tau(\sigma) G(w)\coloneqq\int_{\rdd} G(z) \, \langle  \Opt(\sigma) \pi(z) g,
\pi(w) g \rangle \, dz
\end{equation*}
is continuous from $  L^{p}_{m_1}(\rdd)$ to $L^{p}_{m_2}(\rdd).$

Using \eqref{311}, we see that it is equivalent to prove that the integral operator with kernel
\begin{equation*}
	 K_\tau(z,w)\coloneqq\left|V_{\Phi_\tau} \sigma\left(\cT_\tau(w,z),J(w-z)\right)\right|\frac1{m_1(z)} m_2(w),
\end{equation*}
where $\cT_\tau$ ans $J$ are defined in \eqref{Eq-def-T-J}, is bounded on $L^p(\rdd)$. We do this using the Schur test (see, e.g., \cite[Lemma 6.2.1 (b)]{Grochenig_2001_Foundations}). First we majorize $K_\tau$ with another integral kernel $Q_\tau$ using the condition \eqref{Eq-ConditionWeights} with $w=\phas\in\rdd$ and $z=(y,\eta)\in\rdd$:
\begin{align*}
	K_\tau(z,w)&=\frac{m_2(w) m_0(\cT_\tau(w,z),J(w-z))}{m_1(z)m_0(\cT_\tau(w,z),J(w-z))}\left|V_{\Phi_\tau} \sigma(\cT_\tau(w,z),J(w-z))\right|\\
	&\lesssim \left|V_{\Phi_\tau} \sigma(\cT_\tau(w,z),J(w-z))\right|m_0(\cT_\tau(w,z),J(w-z))\\
	&\eqqcolon Q_\tau(z,w).
\end{align*}

We now show that $Q_\tau$ satisfies the Schur conditions.
In the sequel we adopt the change of variables $w'\equiv w'_{z}(w)\coloneqq J(w-z)$, where $z$ is fixed.
Hence $w=z-J(w')$ and $\cT_\tau(z-J(w'),z)=z-((1-\tau)(Jw')_1,\tau(Jw')_2)$ where $J(w')=((Jw')_1,(Jw')_2)$.
Then:
\begin{multline*}
	\sup_{z\in\rdd}\int_{\rdd}\left|Q_\tau(z,w)\right|\,dw \\
=\sup_{z\in\rdd}\int_{\rdd}\left|V_{\Phi_\tau} \sigma\left(\cT_\tau(z-J(w'),z),w'\right)\right|m_0\left(\cT_\tau(z-J(w'),z),w'\right)\,dw'\\
\leq \int_{\rdd}\sup_{z\in\rdd}\left|V_{\Phi_\tau} \sigma\left(\cT_\tau(z-J(w'),z),w'\right)\right|m_0\left(\cT_\tau(z-J(w'),z),w'\right)\,dw'\\
=\int_{\rdd}\sup_{z\in\rdd}\left|V_\Phi \sigma\left(z,w'\right)\right|m_0\left(z,w'\right)\,dw'
=\norm{\sigma}_{M^{\infty,1}_{m_{0}}}<+\infty.	
\end{multline*}
Now, for every $w$ fixed, define the change of variables $w'\equiv w'_w(z)\coloneqq J(w-z)$, so $z=w+J(w')$ and $\cT_\tau(w,w+J(w'))=w+(\tau(Jw')_1,(1-\tau)(Jw')_2)$. Therefore
\begin{multline*}
\sup_{w\in\rdd}\int_{\rdd}\left|Q_\tau(z,w)\right|\,dz \\
=\sup_{w\in\rdd}\int_{\rdd}\left|V_{\Phi_\tau} \sigma\left(\cT_\tau(w,w+J(w')),w'\right)\right|m_0\left(\cT_\tau(w,w+J(w')),w'\right)\,dw'\\
\leq \int_{\rdd}\sup_{w\in\rdd}\left|V_{\Phi_\tau} \sigma\left(\cT_\tau(w,w+J(w')),w'\right)\right|m_0\left(\cT_\tau(w,w+J(w')),w'\right)\,dw'\\
=\int_{\rdd}\sup_{w\in\rdd}\left|V_{\Phi_\tau} \sigma\left(w,w'\right)\right|m_0\left(w,w'\right)\,dw'
=\norm{\sigma}_{M^{\infty,1}_{m_0}}<+\infty.	
\end{multline*}

Since $K_\tau\lesssim Q_\tau$, it follows that
\begin{equation*}	 \sup_{z\in\rdd}\int_{\rdd}\left|K_\tau(z,w)\right|\,dw<+\infty\qquad\text{and}\qquad\sup_{w\in\rdd}\int_{\rdd}\left|K_\tau(z,w)\right|\,dz<+\infty.
\end{equation*}

Hence from the Schur test it follows that $M_\tau(\sigma)$ is continuous, and due to \eqref{Eq-Proof-ii-2} we notice that
$$
V_g\circ \Opt(\sigma) f=M_\tau(\sigma)\circ V_g f,
$$
where the right hand-side is continuous and takes elements of $\gel(\rd)\subset M^p_{m_1}(\rd)$ into $L^p_{m_2}(\rdd)$. Therefore $\Opt(\sigma)$ is linear, continuous and densely defined. This concludes the proof.
\end{proof}


Schatten class properties for various classes of pseudodifferential operators in the framework of time-frequency analysis are studied by many authors, let us mention just \cite{Grochenig_2001_Foundations, ECSchattenloc2005, MOPf, ToftquasiBanach2017}.
However, for our purposes it is convenient to   recall \cite[Theorem 1.2]{KobMiy2012} about Schatten class property for pseudodifferential operators $\Opt(\sigma)$ with symbols in modulation spaces.

\begin{theorem}\label{Th-Schatten-tau-pseudo}
	Let $\tau\in[0,1]$, $0<p<2$, $d\in\bN$ and
	 \begin{equation}\label{Eq-condition-on-u}
		u>\frac{2d}{p}-d.
	\end{equation}
	Consider $\sigma\in M^2_{m^\tau_u}(\rdd)$, where $m^\tau_u$ is defined as in \eqref{Eq-definition-m-u-tau}. Then
	\begin{equation*}
		\Opt(\sigma)\in S_p(L^2(\rd)).
	\end{equation*}
\end{theorem}

\begin{lemma}\label{Lem-inclusion-for-compactness}
	Let $\tau\in[0,1]$, $\ga\geq1$ and $d\in\bN$. Fix
	\begin{equation*}
		u,s,t>0,\qquad l>u+d,\qquad j\geq u.
	\end{equation*}
	Then
	\begin{equation*}
		M^{\infty,1}_{w^\ga_s\otimes w^\ga_t}(\rdd)\hookrightarrow M^{\infty,1}_{v_l\otimes v_j}(\rdd)\hookrightarrow M^2_{v_u\otimes v_u}(\rdd)\hookrightarrow M^2_{m^\tau_u}(\rdd).
	\end{equation*}
	\begin{proof}
		The first inclusion is due to the inclusion relations between ultra-modulation spaces since $v_l\otimes v_j\lesssim w^\ga_s\otimes w^\ga_t$. The last inclusion follows similarly since $m^\tau_u\lesssim v_u\otimes v_u$, as it is shown in Remark \ref{Rem-m-tau_u-lesssim-v_u}.

		For the second inclusion we use Theorem \ref{Th-GuoChenFanZhao}: $(\i,2,l,u)$ fulfils the condition $(\cC_2)$ and $(1,2,j,u)$ fulfils the condition $(\cC_1)$. This concludes the proof.
	\end{proof}
\end{lemma}

On account of the following corollary all the operators considered in Proposition \ref{Pro-SmoothnessEigenfuctionsWeyl} are compact on $L^2(\rd)$.

\begin{corollary}\label{Cor-compactness}
	Let $\tau\in[0,1]$, $\ga\geq1$ and $s,t>0$. Consider $\sigma\in M^{\infty,1}_{w^\ga_s\otimes w^\ga_t}(\rdd)$. Then $\Opt(\sigma)$ is compact on $L^2(\rd)$.
	\begin{proof}
		The claim follows by Lemma \ref{Lem-inclusion-for-compactness} with $u$ satisfying \eqref{Eq-condition-on-u}, after choosing any $0<p<2$, in addition with Theorem \ref{Th-Schatten-tau-pseudo}.
	\end{proof}
\end{corollary}

Now we prove the decay property of the eigenfunctions of $\Opt(\sigma)$ when the symbol belongs to certain weighted modulation spaces.
This result improves \cite[Proposition 3.6]{BasCorNic2019}, in the sense that we show how faster decay of the symbol implies
stronger regularity and decay properties for the eigenfunctions of the corresponding operator. More precisely,
\cite[Proposition 3.6]{BasCorNic2019} deals with polynomial decay, whereas Theorem \ref{Pro-SmoothnessEigenfuctionsWeyl} allows to consider sub-exponential decay as well.

\begin{theorem}\label{Pro-SmoothnessEigenfuctionsWeyl}
	Fix $\tau\in[0,1]$, $\ga\geq1$ and $s>0$. Consider a symbol \\$\sigma\in M^{\infty,1}_{w^\ga_s\otimes w^\ga_t}(\rdd)$ for every $t$ such that
	\begin{equation*}
		t\geq
		\begin{cases}
		s\tau^{1/\ga}\qquad\qquad&\text{if}\qquad1/2\leq\tau\leq1,\\
		s(1+\tau^2)^{1/2\ga}\quad&\text{if}\qquad0\leq\tau<1/2.
		\end{cases}
	\end{equation*}
	If $\lambda\in (\sigma_P(\Opt(\sigma))\smallsetminus\{0\})\neq\emptyset$, then any $f\in L^2(\rd)$ eigenfunction associated to the eigenvalue $\lambda$ belongs to $\cS^{(\ga)}(\rd)$.
\end{theorem}

\begin{proof}
	We recall \eqref{Eq-condition(sub-)exp-verified} from Lemma \ref{Lem-condition(sub-)exp-verified}: for every $x,\o,y,\eta\in\rd$ we have:
	\begin{equation*}
	\frac{w^\ga_{r'+s'}\phas}{w^{\ga}_{r'}(y,\eta)}\leq w^\ga_{s'}\otimes w^\ga_{t'}\Big(\big((1-\tau)x+\tau y,\tau\o+(1-\tau)\eta\big),\big(\o-\eta,y-x\big)\Big).
	\end{equation*}
	where $s',r'\geq0$ and $t'$ which fulfils \eqref{Eq-condition-on-t}. We consider first the case $1/2\leq\tau\leq1$ and fix $s'=s>0$.

Take $r'=0$, $t\geq s\tau^{1/\ga}$, and apply Theorem \ref{Th-ExtensionWeylOp} with $p=2$, $m_0=w^\ga_s\otimes w^\ga_t$, $m_1=w^\ga_0$ and $m_2=w^\ga_s$ which satisfy \eqref{Eq-ConditionWeights}. Thus $\Opt(\sigma)$ extends to a continuous operator from $M^2_{w^\ga_0}(\rd)=L^2(\rd)$ to $M^2_{w^\ga_s}(\rd)$. Starting with $f\in L^2(\rd)$ we get $f=\lambda^{\inv}\Opt(\sigma) f\in M^2_{w^\ga_s}(\rd)$.

Now, take $r'=s$, $t\geq s+s\tau^{1/\ga}$, and apply Theorem \ref{Th-ExtensionWeylOp} with $p=2$, $m_0=w^\ga_s\otimes w^\ga_t$, $m_1=w^\ga_s$ and $m_2=w^\ga_{2s}$ which satisfy \eqref{Eq-ConditionWeights}. Thus $\Opt(\sigma)$ extends to a continuous operator from $M^2_{w^\ga_s}(\rd)$ to $M^2_{w^\ga_{2s}}(\rd)$, so starting with $f\in M^2_{w^\ga_s}(\rd)$ we get $f=\lambda^{\inv}\Opt(\sigma) f\in M^2_{w^\ga_{2s}}(\rd)$.

Repeating the same argument, and using the inclusion relations between ultra-modulation spaces we obtain:
\begin{equation*}
		f\in\bigcap_{n\in\bN_0}M^2_{w^\ga_{ns}}(\rd)=\bigcap_{k\geq 0}M^2_{w^\ga_{k}}(\rd)=\cS^{(\ga)}(\rd).
	\end{equation*}
	The case $0\leq\tau<1/2$ is done similarly. This concludes the proof.
\end{proof}

We finish the paper with an observation related to localization operators.

Note that by Corollary \ref{Cor-compactness} it follows that the localization operators $\gaw$ in the following statement are compact on $L^2(\rd)$.

\begin{theorem} \label{thm:eigenlocop}
Consider $\ga\geq1$, $s>0$, $a\in M^\infty_{w^\ga_s\otimes1}(\rdd)$ and $\f_1,\f_2\in\gel(\rd)$. If $\lambda\in (\sigma_P(\gaw)\smallsetminus\{0\})\neq\emptyset$, then any $f\in L^2(\rd)$ eigenfunction associated to the eigenvalue $\lambda$ belongs to $\cS^{(\ga)}(\rd)$.
\end{theorem}

\begin{proof}
Since $\f_1,\f_2\in\gel(\rd)$ it follows that $W(\f_2,\f_1)\in\gel(\rdd)\subset M^1_{w^\ga_r\otimes w^\ga_t}(\rdd)$, for every $r,t\geq0$.
We now check that $w^\ga_s\otimes w^\ga_t$ is $w^\ga_r\otimes w^\ga_t$-moderate for every $t\geq0$ and every $r\geq s$:
	\begin{align*}
		w^\ga_s\otimes w^\ga_t(\phas+(y,\eta))&\overset{\text{\eqref{Eq-useful-ineq-1}}}{\leq}w^\ga_s(x)w^\ga_s(y)w^\ga_t(\o)w^\ga_t(\eta)\\
		&\leq w^\ga_r(x)w^\ga_t(\o)w^\ga_s(y)w^\ga_t(\eta)\\
		&=w^\ga_r\otimes w^\ga_t\phas w^\ga_s\otimes w^\ga_t(y,\eta), \;\;\; x,\o,y,\eta\in\rd.
	\end{align*}

We write $\gaw=\Opw(\sigma)$, with $\sigma=a\ast W(\f_2,\f_1)$, and then apply Proposition \ref{Pro-ConvRel} in order to infer $\sigma\in M^{\infty,1}_{w^\ga_s\otimes w^\ga_t}(\rdd)$ for every $t\geq s/2^{1/\ga}$:
	\begin{equation*}
		M^{\infty}_{w^\ga_s\otimes1}(\rdd)\ast M^1_{w^\ga_r\otimes w^\ga_t}(\rdd)\hookrightarrow M^{\infty,1}_{w^\ga_s\otimes w^\ga_t}(\rdd).
	\end{equation*}
The claim  now follows by Theorem \ref{Pro-SmoothnessEigenfuctionsWeyl}.
\end{proof}

\section{Appendix}

\textbf{Proof of Lemma \ref{Lem-condition(sub-)exp-verified}} We first recall that given $0<p\leq q<\infty$ the following holds true:
\begin{equation}\label{Eq-norm_q<norm_p}
	\norm{z}_q=\left(\sum_{i=1}^d\abs{z_i}^q\right)^{\frac1q}\leq\left(\sum_{i=1}^d\abs{z_i}^p\right)^{\frac1p}=\norm{z}_p,
\;\;\;
z=(z_1,\dots,z_d)\in\rd.
\end{equation}
In fact, consider $z$ such that $\norm{z}_p=1$. Hence $\abs{z_i}^p\leq1\,\Rightarrow\,\abs{z_i}\leq1$ for $i=1,\dots,d$. Thus $\abs{z_i}^q\leq\abs{z_i}^p$ and $\sum_{i=1}^d\abs{z_i}^q\leq\sum_{i=1}^d\abs{z_i}^p=1$. Eventually consider $u\in\rd\smallsetminus\{0\}$, then $\norm{u/\norm{u}_p}_q\leq1$ and \eqref{Eq-norm_q<norm_p} is proved.

By using the triangular inequality and \eqref{Eq-norm_q<norm_p} with $q=1$ and $p=\be$, we infer that	for $0<\be\leq1$
		\begin{equation}\label{Eq-useful-ineq-1}
			\abs{\sum_{i=1}^d z_i}^\be\leq\sum_{i=1}^d\abs{z_i}^\be,
\;\;\; z=(z_1,\dots,z_d)\in\rd.
		\end{equation}
Now, by the triangular inequality and \eqref{Eq-useful-ineq-1} with $d=2$ we obtain
\begin{equation}\label{Eq-useful-ineq}
			\abs{x}^\be-\abs{y}^\be\leq\abs{x-y}^\be, \;\;\;
0<\be\leq 1, \;\; x,y\in\rd.
\end{equation}

\par

Next, we observe that for $z,w\in\rd$
\begin{multline*}
\abs{(\tau z,(1-\tau)w)}^2 =\tau^2\abs{z}^2+(1-\tau)^2\abs{w}^2
=\tau^2\abs{z}^2+(\tau^2+1-2\tau)\abs{w}^2\\
=\tau^2(\abs{z}^2+\abs{w}^2)+(1-2\tau)\abs{w}^2
=\tau^2\abs{(z,w)}^2+(1-2\tau)\abs{w}^2\\
\leq
			\begin{cases}
			\tau^2\abs{(z,w)}^2+0\qquad\text{if}\qquad1/2\leq\tau\leq1,\\
			\tau^2\abs{(z,w)}^2+1\abs{w}^2+\abs{z}^2=(1+\tau^2)\abs{(z,w)}^2\qquad\text{if}\qquad0\leq\tau<1/2,
			\end{cases}
\end{multline*}
		which gives
		\begin{equation}\label{Eq-useful-ineq-2}
			\abs{(\tau z,(1-\tau)w)}^{1/\ga}\leq
			\begin{cases}
			\tau^{1/\ga}\abs{(z,w)}^{1/\ga}\qquad\text{if}\qquad1/2\leq\tau\leq1,\\
			(1+\tau^2)^{1/2\ga}\abs{(z,w)}^{1/\ga}\qquad\text{if}\qquad0\leq\tau<1/2.
			\end{cases}
		\end{equation}		
		We can now prove \eqref{Eq-condition(sub-)exp-verified}:
		\begin{align*}
			\frac{w^\ga_{r+s}\phas}{w^{\ga}_r(y,\eta)}&=\exp\left((r+s)\abs{\phas}^{1/\ga}-r\abs{(y,\eta)}^{1/\ga}\right)\\
			&=\exp\left(r\left(\abs{\phas}^{1/\ga}-\abs{(y,\eta)}^{1/\ga}\right)+s\abs{\phas}^{1/\ga}\right)\\
			&\overset{\eqref{Eq-useful-ineq}}{\leq} \exp\left(r\abs{\phas-(y,\eta)}^{1/\ga}+s\abs{\phas}^{1/\ga}\right)\\
			&= \exp(r\abs{(\o-\eta,y-x)}^{1/\ga}+s\abs{\phas}^{1/\ga}-s\abs{(\tau (x- y),(1-\tau)(\o-\eta))}^{1/\ga} ) \\
			&\times \exp(s\abs{(\tau (x- y),(1-\tau)(\o-\eta))}^{1/\ga})\\
			&\overset{\eqref{Eq-useful-ineq}}{\leq} \exp(r\abs{(\o-\eta,y-x)}^{1/\ga}
            +s\abs{\phas-(\tau (x-y),(1-\tau)(\o-\eta))}^{1/\ga}\\
			&+s\abs{(\tau (x-y),(1-\tau)(\o-\eta))}^{1/\ga})\\
			&= \exp(r\abs{(\o-\eta,y-x)}^{1/\ga}+s\abs{((1-\tau)x+\tau y,\tau\o+(1-\tau)\eta)}^{1/\ga}\\
			&+s\abs{(\tau x-\tau y,(1-\tau)\o-(1-\tau)\eta)}^{1/\ga})\\
			&\overset{\eqref{Eq-useful-ineq-2}}{\leq}
			\begin{cases}
			\exp((r+s\tau^{1/\ga})\abs{(\o-\eta,y-x)}^{1/\ga}\\
			\qquad+s\abs{((1-\tau)x+\tau y,\tau\o+(1-\tau)\eta)}^{1/\ga})\quad\text{if}\quad1/2\leq\tau\leq1,\\
			\exp((r+s(1+\tau^2)^{1/2\ga})\abs{(\o-\eta,y-x)}^{1/\ga}\\
			\qquad+s\abs{((1-\tau)x+\tau y,\tau\o+(1-\tau)\eta)}^{1/\ga})\quad\text{if}\quad0\leq\tau<1/2,
			\end{cases}
		\end{align*}
		and Lemma \ref{Lem-condition(sub-)exp-verified} follows from the assumption \eqref{Eq-condition-on-t}.

\textbf{Proof of Lemma \ref{lm:optaustft}}
 Consider $\tau\in(0,1)$ and recast the $\tau$-Winger distribution $W_\tau(\f,f)$ using the operator $\cA_\tau f(t)\coloneqq f\left(\frac{\tau-1}{\tau}t\right)$:
\begin{align*}
W_\tau(\f,f)\phas & =\frac1{\tau^{d}}e^{2\pi i\frac1{\tau}\om x}V_{\cA_\tau f}\f\left(\frac1{1-\tau}x,\frac1{\tau}\om\right)\\
&=\frac1{\tau^{d}}e^{2\pi i\frac1{\tau}\om x}\<\f,M_{\frac1{\tau}\om}T_{\frac1{1-\tau}x}\cA_\tau f\>\\
&=\frac1{\tau^{d}}e^{2\pi i\frac1{\tau}\om x}\<\left(\frac{\tau}{1-\tau}\right)^d\cA_{1-\tau} T_{-\frac1{1-\tau}x}M_{-\frac1{\tau}\om}\f,f\>\\
& =\frac1{\tau^{d}}e^{2\pi i\frac1{\tau}\om x}\overline{\<f,\left(\frac{\tau}{1-\tau}\right)^d\cA_{1-\tau} T_{-\frac1{1-\tau}x}M_{-\frac1{\tau}\om}\f\>}
\end{align*}	
\begin{align*}
&=\frac1{\tau^{d}}e^{2\pi i\frac1{\tau}\om x}\int_{\rdd}\overline{V_gf(z)}\,\overline{\<\pi(z)g,\left(\frac{\tau}{1-\tau}\right)^d\cA_{1-\tau} T_{-\frac1{1-\tau}x}M_{-\frac1{\tau}\om}\f\>}\,dz\\
&=\frac1{\tau^{d}}e^{2\pi i\frac1{\tau}\om x}\int_{\rdd}\overline{V_gf(z)}\,\<\left(\frac{\tau}{1-\tau}\right)^d\cA_{1-\tau} T_{-\frac1{1-\tau}x}M_{-\frac1{\tau}\om}\f,\pi(z)g\>\,dz\\
&=\frac1{\tau^{d}}e^{2\pi i\frac1{\tau}\om x}\int_{\rdd}\overline{V_gf(z)}\,\<\f,M_{\frac1{\tau}\om}T_{\frac1{1-\tau}x}\cA_\tau \pi(z)g\>\,dz\\
&=\int_{\rdd}\overline{V_gf(z)}\,\frac1{\tau^{d}}e^{2\pi i\frac1{\tau}\om x}\<\f,M_{\frac1{\tau}\om}T_{\frac1{1-\tau}x}\cA_\tau\pi(z)g\>\,dz\\
&=\int_{\rdd}\overline{V_gf(z)}\,\frac1{\tau^{d}}e^{2\pi i\frac1{\tau}\om x}V_{\cA_\tau\pi(z)g}\f\left(\frac1{1-\tau}x,\frac1{\tau}\om\right)\,dz\\
&=\int_{\rdd}\overline{V_gf(z)}\,W_\tau(\f,\pi(z)g)\phas\,dz.
\end{align*}
Therefore
\begin{multline*}
\<\Opt(\sigma) f,\f\>=\<\sigma,W_\tau(\f,f)\>=\<\sigma,\int_{\rdd}\overline{V_gf(z)}\,W_\tau(\f,\pi(z)g)\phas\,dz\>
\\
=\int_{\rdd}V_gf(z)\,\<\sigma,W_\tau(\f,\pi(z)g)\phas\>\,dz
=\int_{\rdd}V_gf(z)\,\<\Opt(\sigma)(\pi(z)g),\f\>\,dz
\end{multline*}
and \eqref{Eq-Proof-ii} holds true when $\tau\in(0,1)$.

For the cases $\tau=0,1$ we need the operator $J$ defined in \eqref{Eq-def-T-J} and the following equalities which come from easy computations (cf. \cite{Grochenig_2001_Foundations}):
\begin{equation*}
	V_gf\phas=e^{-2\pi ix\o}V_{\hg}\hf(\o,-x),\,\cF T_x=M_{-x}\cF,\,\cF M_\o=T_\o\cF,\,T_xM_\o=e^{-2\pi ix\o}M_\o T_x.
\end{equation*}

Therefore \eqref{Eq-Proof-ii} is proved for $\tau=0,1$ in the following manner. We put $z=\phas$ and let $\sigma$ acts
on functions of variables $(y,\eta)$:
\begin{align*}
	\<\Opz(\sigma)f,\f\> & =\<\sigma,W_0(\f,f)\>  =\<\sigma,e^{-2\pi iy\eta}\f(y)\overline{\hf(\eta)}\>\\
	&=\<\sigma,e^{-2\pi iy\eta}\f(y)\overline{\intrdd V_{\hg}\hf(z')\pi(z')\hg(\eta)\,dz'}\>\\
	&=\<\sigma,e^{-2\pi iy\eta}\f(y)\overline{\intrdd V_{\hg}\hf(Jz)\pi(Jz)\hg(\eta)\,dz}\>\\
	&=\<\sigma,e^{-2\pi iy\eta}\f(y)\overline{\intrdd V_{g}f(z)e^{2\pi ix\o}\pi(Jz)\hg(\eta)\,dz}\>\\
\end{align*}
\begin{align*}
	&= \<\sigma,\intrdd \overline{V_{g}f(z)}e^{-2\pi iy\eta}\f(y)\overline{e^{2\pi ix\o}\pi(Jz)\hg(\eta)}\,dz\>\\
	&=\intrdd V_{g}f(z)\<\sigma,e^{-2\pi iy\eta}\f(y)\overline{e^{2\pi ix\o}\pi(Jz)\hg(\eta)}\>\,dz\\
	&=\intrdd V_{g}f(z)\<\sigma,e^{-2\pi iy\eta}\f(y)\overline{\widehat{\pi(z)g}(\eta)}\>\,dz\\
	&=\intrdd V_{g}f(z)\<\sigma,W_0(\f,\pi(z)g)\>\,dz\\
	&=\intrdd V_{g}f(z)\<\Opz(\sigma)\pi(z)g,\f\>\,dz.\\
\end{align*}
The case $\tau = 1$, i.e.
$$ \displaystyle
	\<\Opo(\sigma)f,\f\> = \intrdd V_gf(z)\<\Opo(\sigma)\pi(z)g,\f\>\,dz,
$$
can be proved in the same manner. The details are left to the reader.


\section*{Acknowledgments}
F. Bastianoni is member of the Gruppo Nazionale per l'Analisi Matematica, la Probabilit\`a e le loro Applicazioni (GNAMPA) of the Istituto Nazionale di Alta Matematica (INdAM).
\par
The authors would like to thank Elena Cordero and Fabio Nicola for fruitful conversations and comments.

\end{document}